\documentclass[11pt]{amsart}
\usepackage{palatino, mathpazo}
\usepackage[mathscr]{eucal}
\usepackage{amssymb}
\usepackage[all]{xy}
\usepackage{graphicx}
\usepackage{texdraw}
\usepackage{color}

\newtheorem{theorem}{Theorem}[section]
\newtheorem{lemma}[theorem]{Lemma}
\newtheorem{proposition}[theorem]{Proposition}
\newtheorem{corollary}[theorem]{Corollary}
\newtheorem{conjecture}[theorem]{Conjecture}
\newtheorem{definition}[theorem]{Definition}
\theoremstyle{remark}
\newtheorem{remark}[theorem]{Remark}
\newtheorem{example}[theorem]{Example}
\newtheorem{question}[theorem]{Question}
\numberwithin{equation}{section}

\newcommand{\p}{\partial}
\newcommand{\z}{{\bf z}}
\newcommand{\bz}{{\bf z}}
\newcommand{\T}{\vartheta}
\newcommand{\f}{{\bf f}}
\newcommand{\w}{{\bf w}}

\begin{document}

\title{Algebraic methods in periodic singular Liouville equations}
\author{Chin-Lung Wang}
\address{Department of Mathematics, National Taiwan University, Taipei}
\email{dragon@math.ntu.edu.tw}

\maketitle

\begin{abstract}
We explain how algebraic geometry comes into play in the study of non-linear mean field (singular Liouville) equations 
$$
\triangle u + e^u = 4\pi \sum_{i = 1}^N \ell_i \delta_{p_i}
$$ 
on a flat torus $E = \Bbb C/\Lambda$, where $N, \ell_1, \ldots, \ell_N \in \Bbb N$, $p_i \in E$ are distinct points, and $\delta_{p_i}$ is the Dirac measure at $p_i$. 

The case with one singular source ($N = 1$) had been studied extensively in recent years. We start with a survey of this case with emphasizes on the constructions of Lam\'e curves $\overline X_n$ and pre-modular forms $Z_n(\sigma, \tau)$ which encodes the structure of solutions of the PDE. 

We then discuss extensions to the case of general $N$. The basic tool is the monodromy theory for generalized Lam\'e equations. Two aspects are discussed: (1) For $\ell := \sum_{i = 1}^N \ell_i$ being odd, an exact counting formula of \emph{algebraic degree} is proved. (2) For $\ell$ being even, the existence of generalized Lam\'e curves parametrizing logarithmic-free solutions is proposed.
\end{abstract}

\tableofcontents

\setcounter{section}{-1}

\section{Introduction} \label{mlame}
\setcounter{equation}{0}

\subsection{Mean field equations and generalized Lam\'e equations}

We study mean field (singular Liouville) equations with multiple singular sources on a flat torus $E = \Bbb C/\Lambda$, $\Lambda = \Bbb Z \omega_1 + \Bbb Z \omega_2$:
\begin{equation} \label{MFE-ln}
\triangle u + e^u = 4\pi \sum_{i = 1}^{N} \ell_i\delta_{p_i} \quad \text{on $E$},
\end{equation}
with local singular strength $\ell_i \in \Bbb N$ at $p_i \in E$, $1\le i \le N$, and $p_i \ne p_j$ if $i \ne j$. Denote by $L := \sum_{i = 1}^N \ell_i p_i$ the divisor of singular source, and $\ell := \deg L = \sum_{i = 1}^N \ell_i$ the total singular strength. 

We start by recalling some basic results on equation \eqref{MFE-ln} (cf.~\cite[\S 1.1]{CLW}). A (local) developing map $f$ of a solution $u$ to equation \eqref{MFE-ln} is a local meromorphic function away from $L$ which is related to $u$ by 
\begin{equation} \label{e:dev}
u = 8\pi + \log \frac{|f'|^2}{(1 + |f|^2)^2}.
\end{equation}
A classical theorem of Liouville says that any solution $u$ is locally represented by \eqref{e:dev} for some $f$. Moreover, the choice of $f$ is unique up to the obvious ${\rm PSU}(2)$ (Mobius) action
$$
f \mapsto Mf := \frac{pf - \bar q}{qf + \bar p}, \qquad M = \begin{pmatrix} p & -\bar q \\ q & \bar p \end{pmatrix} \in {\rm SU}(2)/{\pm 1}.
$$

The integral condition $\ell_i \in \Bbb N$ for all $i$ implies that $f$ extends to a global meromorphic function on $\Bbb C$. In particular 
$$
f(z + \omega_1) = S_1 f(z), \qquad f(z + \omega_2) = S_2 f(z),
$$ 
for some $S_i \in {\rm PSU}(2)$. Then $f$ is called a type I developing map if
\begin{equation} \label{type-I}
\begin{split}
f(z + \omega_1) &= -f(z), \\
f(z + \omega_2) &= \frac{1}{f(z)},
\end{split}
\end{equation}
and a type II developing map if there are $\theta_i, \theta_2 \in \Bbb R$ such that
\begin{equation} \label{type-II}
\begin{split}
f(z + \omega_1) &= e^{i\theta_1} f(z), \\ 
f(z + \omega_2) &= e^{i\theta_1} f(z).
\end{split}
\end{equation}
From $S_1 S_2 = S_2 S_1$ in ${\rm PSU}(2)$, by choosing $f$ suitably, one sees easily that every solution $u$ belongs to exactly one of these two types. We then call such an $f$ a normalized developing map of $u$. \smallskip

This integrability structure of solutions is best explained by its associated ODE of equation \eqref{MFE-ln}, which is a \emph{generalized Lam\'e equation}: 
\begin{equation} \label{e:gLame}
w'' - \Big(\sum_{i = 1}^N \eta_i(\eta_i + 1) \wp(z - p_i) + \sum_{i = 1}^N A_i \zeta(z - p_i) + B\Big) w = 0,
\end{equation}
where the potential term is the Schwarzian derivative $S(f)$ (c.f.~\eqref{g-mS-der}) with 
$$
\eta_i = \tfrac{1}{2}\ell_i \in \tfrac{1}{2} \Bbb N, \qquad \sum_{i = 1}^N A_i = 0, \quad A_i, B \in \Bbb C.
$$ 

The ${\rm PSU}(2)$ action corresponds to the requirement that equation \eqref{e:gLame} has \emph{projective unitary monodromy group}. We will show that the type I case corresponds to the case with $\ell$ being odd and equation \eqref{e:gLame} has finite monodromy $M$, or equivalently the projective monodromy group is the Klein-four group $PM \cong K_4$. The type II case corresponds to the case with $\ell$ being even and equation \eqref{e:gLame} has unitary monodromy. The ratio $f = w_1/w_2$ of its two independent solutions then gives a developing map. \smallskip
\subsection{Known results for $N = 1$}

Equation \eqref{MFE-ln} for $N = 1$ ($\ell = \ell_1$) had received significant progresses in the last decade through a joint effort on non-linear analysis and its corresponding classical Lam\'e equation
\begin{equation} \label{e:Lame}
w'' = (\eta (\eta + 1) \wp(z) + B) w
\end{equation}
with $\eta = \ell/2$ \cite{LW, CLW, LW-II}. The Lam\'e equation \eqref{e:Lame} (for all $\eta \in \Bbb R_{\ge 0}$) had been studied extensively since the 19th century, though a full account on its monodromy theory is still awaiting. As a byproduct, advances on it had also been made through such interactions. 

Let $\wp'(z)^2 = 4 \wp(z)^3 - g_2 \wp(z) - g_3$ be the Weierstrass equation of the torus $E = \Bbb C/\Lambda$, with $g_2 = 60 E_4(\Lambda)$ and $g_3 = 140 E_6(\Lambda)$ be the standard modular functions (Eisenstein series). 

For $\ell = 2n + 1$ being odd, $n \in \Bbb Z_{\ge 0}$, the classical result due to Brioschi, Halphen and Crawford says that there is a universal (weighted homogeneous) polynomial $p_n(B; g_2, g_3)$ of degree $n + 1$ whose roots $B$'s correspond to those equations in \eqref{e:Lame} which admit only log-free solutions (cf.~the proof in \cite[Theorem 3.1]{CLW} and the references in \cite[\S 23.7]{Whittaker}). Hence they have finite monodromy $PM \cong K_4$. A new proof to this result, as well as the explicit construction of the developing maps $f$ were carried out in \cite[\S 3.4]{CLW}. It was also announced in \cite[Remark 3.2.1]{CLW} that the idea of this new proof could be generalized to study the case with multiple singularities (cited as reference [11] there). Due to the increasing diversities of techniques, the writing of that paper was postponed. Nevertheless the promised generalization is now presented in this paper (see \S \ref{ss:N} below).   

For $\ell = 2n$ being even, on the contrary, there is a polynomial $\ell_n(B; g_2, g_3)$ of degree $2n + 1$ such that \eqref{e:Lame} admits only log-free solutions if and only if $\ell_n(B; g_2, g_3) \ne 0$. An eigenvalue problem $Lw := w'' - I w = B w$ of an ODE with potential $I$ satisfying such a property is known as a \emph{finite gap potential}. The integral Lam\'e equations with $I = n(n + 1) \wp(z)$ provides the first non-trivial such examples. Since there is at least one log-free solution, the set of log-free solutions are thus parametrized by the \emph{Lam\'e curve} $Y_n: C^2 = \ell_n(B; g_2, g_3)$ which is an hyperelliptic curve of arithmetic genus $g = n$ under the projection $\pi: Y_n \to \Bbb C$, $(C, B) \mapsto B$. 

The finer structure of $Y_n$ had been analyzed in details in \cite{CLW} through both analytic and algebraic methods. It allows us to go on to characterize the loci where \eqref{e:Lame} has unitary monodromy. It starts with the \emph{Hermite--Halphen ansatz} \cite[p.495--498]{Halphen}: for $a = (a_1. \ldots, a_n) \in \Bbb C^n$,
\begin{equation}
w_a := e^{z \sum \zeta(a_i)} \prod_{i = 1}^n \frac{\sigma(z - a_i)}{\sigma(z)}.
\end{equation}
Then $w_a$, $w_{-a}$ provide independent solutions to \eqref{e:Lame} if and only if $[a] := a \pmod{\Lambda} \in {\rm Sym}^n E$ satisfies the following system of algebraic equations
\begin{equation} \label{e:Xn}
\sum_{i = 1}^n y_i x_i^k = 0, \qquad 0 \le k \le n - 2,
\end{equation}
where $(x_i, y_i) := (\wp(a_i), \wp'(a_i))$. This defines the \emph{Liouville curve} $X_n \subset {\rm Sym}^n E$ which can be identified with the unramified loci of $\pi: Y_n \to \Bbb C$. One crucial result proved in \cite[\S 7.6]{CLW} says that $\overline X_n \subset {\rm Sym}^n E$ coincides with the projective hyperelliptic model of $Y_n$. In particular, $0^n \in \overline X_n$ is a non-singular point. The unitary constraint, which is the type II constraint \eqref{type-II} in this case, can be written as a Green function equation on $X_n$:
\begin{equation} \label{e:green}
\sum_{i = 1}^n \nabla G(a_i) = 0, \qquad a \in X_n,
\end{equation}
where $G$ is the Green function on $E$ centered at $0 \in E$.
 
Then in \cite[Theorem 3.2]{LW-II}, it was shown that the addition map $\sigma_n: \overline X_n \to E$ defined by $\sigma_n(a) = \sum_{i = 1}^n a_i$ is a branched cover of degree $\tfrac{1}{2} n(n + 1)$. Based on $\sigma_n$ and the Hecke function (see \S \ref{modular}), a pre-modular form $Z_n(\sigma, \tau)$ was constructed which has the property that every non-trivial zero $(\sigma, \tau)$ (i.e~$\sigma \not\in E_\tau[2]$) corresponds to a unique solution $a$ to equations \eqref{e:Xn} and \eqref{e:green} via $\sigma_n(a) = \sigma$. In particular, $f := w_a/w_{-a}$ is a normalized developing map of a type II solution $u$ to the mean field equation \eqref{MFE-ln} with $N = 1$.

I will review this procedure in more details in \S \ref{s:HE}, with emphasizes on the problem on explicit constructions of $Z_n(\sigma, \tau)$.

\subsection{Results and proposals for general $N$} \label{ss:N}

The second aim of this paper is to propose extensions of some results in $N = 1$ to the case $N \ge 2$.

\begin{theorem} \label{t:main-1}
When $\ell = \deg L$ is odd, there is only a finite number of solutions to equation \eqref{MFE-ln}. All the solutions are of type I and algebraically integrable. Moreover, its algebraic degree $d_L$ is given by the formula:
\begin{equation} \label{degree}
d_{L} = \tfrac{1}{2} \prod_{i = 1}^N (\ell_i + 1) \in \Bbb N.
\end{equation}
(The expression is an integer since some $\ell_i$ is odd). 
\end{theorem}

That is, all the solutions can be effectively constructed by way of solving certain explicit polynomial equations. This refines a previous result in \cite{CL2} which says that the topological Leray--Schauder degree of eqution \eqref{MFE-ln} is defined and given by \eqref{degree}. The proof of Theorem \ref{t:main-1} is given in Theorem \ref{t:degree} and Corollary \ref{c:type-I}, whose basic idea is explained below.

For all $\ell \in \Bbb N$, we first construct an explicit polynomial system on $A_i$ and $B$ whose zeros correspond to equations \eqref{e:gLame} with only log-free solutions. 

When $\ell$ is odd, the polynomial system has isolated zeros $(\{A_i\}, B)$'s and each of them satisfies $PM \cong K_4$ automatically. Thus the correspondence between equations \eqref{MFE-ln} and \eqref{e:gLame} is exact and the problem is completely reduced to the study of the (affine) polynomial system.

Counting roots of affine polynomial equations is in general not easy due to lacking of an \emph{affine Bezout theorem}. We projectivize the system and subtract the infinity contribution from the Bezout degree. It turns out that the infinity point $Q$ is isolated and we need only compute its multiplicity. Here a trick using confluent hypergeometric equation and its Kummer solution is employed to determine the top degree terms $q_{\ell_i}$ of the polynomial equations in $(\{A_i\}, B)$ (see Lemma \ref{top-term-lem}):
\begin{equation} \label{e:top}
q_{\ell_i} = \frac{(-1)^\ell_i}{(\ell_i !)^2} \prod_{j = 0}^{\ell_i} (A_i - (\ell_i - 2j) B^{1/2}), \qquad i = 1, \ldots, N.
\end{equation}
This leads easily to a proof of Theorem \ref{t:main-1}. 
\smallskip

When $\ell$ is even, the situation is more involved, and we restrict ourselves to the primitive case $\ell_i = 1$ for all $1 \le i \le N = 2n = \ell$. Based on \eqref{e:top}, the zero set $V$ of the polynomial system on $(\{A_i\}, B)$ is shown to consist of a finite number of points and complex algebraic curves. 

At this point, I conjecture that, even for the non-primitive case, $V$ does contain non-trivial curve component $V_0$ (cf.~Conjecture \ref{c:conj}). The conjecture is clear if the set of singular sources $\{p_i\}$ is symmetric with respect to a center point $o \in E$:
$$
\{p_i - o\} = \{-(p_i - o)\}.
$$ 
In such a case each irreducible curve component of $V$ consists of symmetric log-free parameter $(\{A_i\}, B)$ in the sense that 
$$
A_{n + i} = -A_i, \qquad i = 1, \ldots, n
$$ 
such that the corresponding generalized Lam\'e equation descents to an ODE on $\Bbb P^1$ under the elliptic projection map $E \to \Bbb P^1$ (cf.~Example \ref{l=2n}).

\begin{proposition} \label{p:sym}
In the primitive case, the conjecture on the existence of non-trivial curve component holds for the cases $\ell = 2, 4$, i.e.~$n = 1, 2$.  
\end{proposition}

The case $n = 1$ is trivially true. The case $n = 2$ is proved in Example \ref{e:n=2} by observing a factorization formula of elliptic functions
$$
\wp_{24} - \wp_{13} = (\zeta_{24} - \zeta_{13} + \zeta_{12} - \zeta_{34}) (\zeta_{24} + \zeta_{13} - \zeta_{14} - \zeta_{23}),
$$
where $\wp_{ij} := \wp(p_i - p_j)$ and $\zeta_{ij} := \zeta(p_i - p_j)$. 

Based on the conjecture we further propose the existence of a two to one ramified cover $Y_L \to V_0$ which parametrizes log-free solutions. We call $Y_L$ the \emph{generalized Lam\'e curve} associated to the even degree divisor $L$. Also we conjecture the existence of a correspondence between solutions of equation \eqref{MFE-ln} (i.e.~$(\{A_i\}, B)$ with unitary monodromy) and the non-trivial zeros of certain pre-modular forms analogous to the case $N = 1$.  

\subsection{Structure of the paper}

This paper is an expanded version of the talk I gave at the first ICCM Annual Meeting on December 28, 2017, where I explained joint works with C.-L.~Chai and C.-S.~Lin on \emph{mean field equations, hyperelliptic curves and modular forms} \cite{CLW, LW-II}, as well as part of results in \cite{CKLW} concerning with \emph{a complete understanding of the geometry of critical points of Green's function on tori} which confirmed the conjecture posed in \cite{LW}. 

The first two sections (\S \ref{modular}, \S \ref{s:HE}) contains exactly the material presented in my talk, where all the results are connected to the case of mean field equations or Lam\'e equations with one singular source ($N = 1$). 

\S \ref{modular} is mainly for the case $N = 1$ and $n = 1$ where Green functions and Hecke functions play the main role. \S \ref{s:HE} is on the hyperelliptic geometry of Lam\'e curves and constructions of the pre-modular forms. These lay the foundation of the general cases $n \in \Bbb N$ (and $N = 1$).  

The remaining three sections (\S \ref{s:poly-sys} -- \S \ref{IFT}) are attempts to extend the theory to equations with multiple singular sources ($N \in \Bbb N$). 

Specific results are (1) Theorem \ref{t:degree} on the algebraic degree counting formula (for $\ell$ odd), (2) Proposition \ref{PMK4} and Theorem \ref{t:finite} on finite monodromy groups (for $\ell$ odd), and (3) Example \ref{l=2n} and Example \ref{e:n=2} on the existence of generalized Lam\'e curves in the primitive case (when $\{p_i\}$ admits a center or when $\ell = 4$ and there is no restriction on $\{p_1, p_2, p_3, p_4\}$). 

In the appendix I explain the subtlety of explicit constructions of the pre-modular forms $Z_n$ and how classical approach fails when $n \ge 4$.

\section{Green functions and Hecke functions} \label{modular}
\setcounter{equation}{0}

\subsection{Green functions on tori \cite{LW}}
The Green function $G(z, w)$ on $E = 
\mathbb{C}/\Lambda$, $\Lambda = \mathbb{Z}\omega_1 +
\mathbb{Z}\omega_2$ is the unique function on $E\times E$ which
satisfies
$$
-\triangle_z G(z, w) = \delta_{w}(z) - \frac{1}{|E|}
$$
and $\int_E G(z, w)\,dA = 0$. Translation invariance of $\triangle_z$ implies $G(z,w) = G(z - w, 0)$ and it is enough to consider $G(z) := G(z, 0)$. Asymptotically
$$
G(z) = -\frac{1}{2\pi}\log|z| + O(|z|^2).
$$
As expected, $G$ can be explicitly solved in terms of elliptic functions. Let $z = x + iy$, $\tau := \omega_2/\omega_1 = a + ib
\in \mathbb{H}$ and $q = e^{\pi i\tau}$ with $|q| = e^{-\pi b} <
1$. We have the odd theta function
$$
\T_1(z; \tau) := -i\sum_{n = -\infty}^\infty (-1)^n q^{(n +
\frac{1}{2})^2}e^{(2n + 1)\pi iz}.
$$
Then on $E_\tau$ (notice the $\tau$ dependence),
$$
G(z; \tau) = -\frac{1}{2\pi}\log\left|\frac{\T_1(z; \tau)}{\T'_1(0; \tau)}\right| +
\frac{1}{2b}y^2 + C(\tau)
$$
for some constant $C(\tau)$ depending only on $\tau$. Then
\begin{equation} \label{e:Gz}
\nabla G(z) = 0 \Longleftrightarrow \frac{\p G}{\p z} \equiv \frac{-1}{4\pi}\left((\log\T_1)_z + 2\pi
i \frac{y}{b}\right) = 0.
\end{equation}

Now we translate equation \eqref{e:Gz} into Weierstrass theory. Recall the Weierstrass elliptic and quasi-elliptic functions with periods $\Lambda$ \cite{Whittaker}:
\begin{equation*}
\begin{split}
\wp(z) &:= \frac{1}{z^2} + \sum_{\omega \in \Lambda^\times} \Big(\frac{1}{(z - \omega)^2} - \frac{1}{\omega^2}\Big), \\
\zeta(z) &:= -\int^z \wp = \frac{1}{z} + \sum_{\omega \in \Lambda^\times} \big( \frac{1}{z - \omega} + \frac{1}{\omega} + \frac{z}{\omega^2} \Big), \\
\sigma(z) &:= \exp \int^z \zeta(w)\,dw = z + \cdots. 
\end{split}
\end{equation*}
The function $\sigma$ is entire, odd with a simple zero on lattice points and
$$
\sigma(z + \omega_i) = -e^{\eta_i(z +
\frac{1}{2}{\omega_i})}\sigma(z),
$$
where $\eta_i = \zeta(z + \omega_i) - \zeta(z) =
2\zeta(\frac{1}{2}\omega_i)$ are the quasi-periods. Indeed, $\sigma(z)$ is related to the theta function by
$$
\sigma(z) = e^{\eta_1 z^2/2}\frac{\T_1(z)}{\T_1'(0)}.
$$
Hence the main term in the right-hand-side of \eqref{e:Gz} is
$$
(\log\T_1(z))_z = \zeta(z) - \eta_1 z.
$$

For simplicity we set $\omega_1 = 1$, $\omega_2 = \tau = a + bi$, $\omega_3 = \omega_1 + \omega_2$, and 
$$
z = x + yi = r\omega_1 + s\omega_2 = (r + sa) + sbi.
$$ 
By Legendre's relation $\eta_1 \omega_2-\eta_2 \omega_1 = 2\pi i$ we compute
\begin{equation*}
\begin{split}
(\log \T_1)_z + 2\pi i \frac{y}{b} &= (\zeta(z) - \eta_1 z) + 2\pi i s \\
&=\zeta(z) - \eta_1 r - \eta_1 s \omega_2 + (\eta_1 \omega_2 - \eta_2)s \\
&= \zeta(z) - r \eta_1 - s \eta_2.
\end{split}
\end{equation*}
Hence $\nabla G(z) = 0$ if
and only if
\begin{equation} \label{e:Hecke}
Z := -4\pi G_z = \zeta(r\omega_1 + s\omega_2) - (r\eta_1
+ s\eta_2) = 0.
\end{equation}

\begin{question} [Structure of critical points] \label{q:critical}
How many critical points can $G$ have in
$E$? What is its dependence in $\tau \in \Bbb H$?
\end{question}

The 3 half-periods are trivial critical points. Indeed,
$$
G(z) = G(-z) \Longrightarrow \nabla G(z) = -\nabla G(-z).
$$
Let $p = \frac{1}{2}{\omega_i}$ then $p = -p$ in $E$ and so
$\nabla G(p) = -\nabla G(p) = 0$. Other critical points must appear in pair $\pm z \in E$.

\begin{example} [Maximal principle] \label{e:rect} 
For rectangular tori $E$: $(\omega_1,
\omega_2) = (1, \tau = bi)$, the half-periods $\frac{1}{2}\omega_i$, $i = 1, 2, 3$
are precisely all the critical points.
\end{example}

\begin{example} [$\Bbb Z_3$-symmetry] For the 60 degree torus $E$ with $\tau = \rho := e^{\pi
i/3}$, it is easy to check that there are (at least) 2 more critical points 
$$
p = \tfrac{1}{3}\omega_3, \qquad -p = -\tfrac{1}{3}\omega_3 \equiv \tfrac{2}{3}\omega_3.
$$
\end{example}

Question \ref{q:critical} is indeed the starting point of the whole project around 2003. In \cite{LW} we showed that there are at most 5 critical points and conjectured the $\tau$ dependence in a precise manner. The conjecture was recently solved in \cite{CKLW} and I will give a sketch on it in this section.

\subsection{Periodic singular Liouville equations}

Geometry of the Green function $G$ plays a
fundamental role in the non-linear mean field equations. On a flat torus $E$ it
takes the form
\begin{equation} \label{e:mfe1}
\triangle u + e^u = \rho \delta_0, \qquad \rho \in \Bbb R_+.
\end{equation}
(The name comes from the fact that It is the mean field limit of Euler flow in statistic physics.) 

When $\rho \not\in 8\pi \mathbb{N}$, it was proved in \cite{CL2} that the Leray-Schauder degree is
$$
d_\rho = n + 1 \quad \mbox{if} \quad 8 n \pi < \rho < 8(n +
1)\pi.
$$
The degree is independent of the shape (moduli) of $E$. The interesting cases (critical values) are $\rho = 8\pi n$ where the degree theory fails completely.

\begin{theorem}[Existence criterion via $\nabla G$ for $n = 1$] \cite[Theorem 1.1]{LW} {\ }

For $\rho = 8\pi$, the mean field equation on $E = \mathbb{C}/\Lambda$ in \eqref{e:mfe1} has solutions if and only if $G$ has more than 3 critical
points. Moreover, each extra pair of critical points $\pm p$
corresponds to an one parameter family of solutions $u_\lambda$,
where 
$$
\lim_{\lambda \to \infty} u_\lambda(z) \quad \mbox{and} \quad \lim_{\lambda \to -\infty} u_\lambda(z)
$$ 
blow up precisely at $z \equiv \pm p$.
\end{theorem}

The purpose of this section is to explain the above correspondence in a more general setting for all $n \in \Bbb N$ obtained in \cite{CLW}.

\subsubsection{Structure of solutions.}

Liouville's theorem says that any solution $u$ of
$\triangle u + e^u = 0$ in a simply connected domain $D
\subset \mathbb{C}$ is of the form
$$
u = \log \frac{8|f'|^2}{(1 + |f|^2)^2},
$$
where $f$, called a developing map of $u$, is meromorphic in $D$.

It is straightforward to show that for $\rho = 8\pi \eta \in \Bbb R$,
\begin{equation} \label{e:S(f)}
S(f) \equiv \frac{f'''}{f'} - \frac{3}{2}
\left(\frac{f''}{f'}\right)^2 = u_{zz} - \frac{1}{2} u_z^2 = -2\eta(\eta + 1)\frac{1}{z^2} + O(1).
\end{equation}
That is, any developing map $f$ of $u$ has the same Schwartz
derivative $S(f)$, which is elliptic on $E$. Hence there is a $B \in \Bbb C$ such that 
$$
S(f) = - 2(\eta(\eta + 1) \wp(z) + B).
$$

By the theory of ODE, locally $f = w_1/w_2$ for two solutions $w_i$ of the Lam\'e equation $L_{\eta, B}\, y = 0$:
\begin{equation} \label{e:lame}
y'' +\frac{1}{2} S(f) y = y'' - (\eta(\eta + 1) \wp(z) + B) y = 0.
\end{equation}
Furthermore, for any two developing maps $f$ and $\tilde f$ of
$u$, there exists $S = \begin{pmatrix} p & -\bar q\\ q & \bar
p\end{pmatrix}\in PSU (2)$ such that 
$$
\tilde f = Sf := \frac{pf - \bar q}{qf + \bar p}.
$$
So we conclude that  solutions to the mean field equation correspond to Lam\'e equations with \emph{unitary projective monodromy groups}. 

Geometrically the Liouville equation is simply the
prescribing Gauss curvature equation in the new metric $g = e^w
g_0$ over $D$, where $w = u/2 - \log \sqrt{2}$ and $g_0$ is the Euclidean flat metric on
$\mathbb{C}$:
\begin{equation} \label{Gauss}
K_g = -e^{-u} \triangle u = 1.
\end{equation}
It is then clear the inverse stereographic projection
$\mathbb{C} \to S^2$
\begin{equation*}
(X, Y, Z) = \Big( \frac{2x}{1 + x^2 + y^2},
\frac{2y}{1 + x^2 + y^2}, \frac{-1 + x^2 + y^2}{1 + x^2 +
y^2}\Big)
\end{equation*}
provides solutions to (\ref{Gauss}) with conformal factor
\begin{equation*}
e^w = e^{\frac{1}{2} u - \frac{1}{2}\log 2} = \frac{2}{1 + |z|^2}.
\end{equation*}

Starting from this special solution for $D = \Delta$,
the unit disk, general solutions on simply connected domain $D$
can be obtained by using the Riemann mapping theorem via a
holomorphic map $f: D \to \Delta$. The conformal factor is then the one as expected:
\begin{equation*}
e^u = \frac{8|f'|^2}{(1 + |f|^2)^2}.
\end{equation*}

The problem is to glue the local developing maps to a ``global one''. This is a monodromy problem on the once punctured torus $E^\times = E \backslash \{0\}$. Since it  is homotopic to ``8'', we have for $x_0 \in E^\times$,
$$
\pi_1(E^\times, x_0) = \Bbb Z * \Bbb Z
$$ 
is a free group of rank two. 

\begin{lemma} [Developing map for $\eta = \tfrac{1}{2} \ell \in \tfrac{1}{2}\Bbb Z$]
Given $\Lambda$, for $\rho = 4\pi \ell$, $\ell \in
\mathbb{N}$, by analytic continuation across $\Lambda$, $f$ is glued into a meromorphic
function on $\mathbb{C}$. (Instead of on $E = \Bbb C/\Lambda$.)
\end{lemma}

\begin{proof}
The indicial equation of the Lam\'e equation \eqref{e:lame} is $\lambda(\lambda - 1) = \eta(\eta + 1)$, hence the two local solutions around a lattice point takes the form $w_1 = z^{\eta + 1} g_1(z)$, $w_2 = z^{-\eta} g_2(z)$ with $g_i$ holomorphic and $g_i(0) \ne 0$, and then $f = z^{2\eta + 1} g_1(z)/g_2(z)$ which is meromorphic at $z = 0$.
\end{proof}

Now there are two type of constraints on $f$: global and local.
\begin{itemize}
\item[$\bullet$]
\emph{First constraint from the double periodicity:}
$$
f(z + \omega_1) = S_1 f,\quad f(z + \omega_2) = S_2 f
$$
with $S_1S_2 = \pm S_2 S_1$ (abelian projective monodromy).

\item[$\bullet$]
\emph{Second constraint from the Dirac singularity:}
\begin{itemize}
\item [(1)] If $f(z)$ has a zero/pole at $z_0 \not\in {\Lambda}$ then order $r = 1$. \medskip

\item [(2)] $f(z) = a_0 + a_{\ell + 1} (z - z_0)^{\ell + 1} + \cdots$ is
regular at $z_0 \in {\Lambda}$.
\end{itemize}
\end{itemize}

\subsubsection{Type I (Topological) Solutions $\Longleftrightarrow$ $\ell = 2n + 1$} This means that 
\begin{equation*}
f(z + \omega_1) = -f(z), \qquad f(z + \omega_2) = \frac{1}{f(z)}.
\end{equation*}
Then the loagarithmic derivative
$$
g = (\log f)' = {f'}/{f}
$$ 
is elliptic on the doubled torus $E' = \Bbb C/\Lambda'$ where $\Lambda' = \Bbb Z\omega_1 + \Bbb Z 2\omega_2$, with the only zeros at $z_0 \equiv 0 \pmod \Lambda$ of order $\ell = 2n + 1$. Using Weierstrass functions on $E'$, $g(z)$ takes the form
\begin{equation*} \label{g-zeta}
g(z) = \sum_{i = 1}^\ell (\zeta(z - p_i) - \zeta(z - p_i - \omega_2))
+ c.
\end{equation*} 

The equations $0 = g(0) = g''(0) = g^{(4)}(0) = \cdots$ implies that $f$ is an even function (by a non-trivial symmetric function argument \cite[\S 2.1]{CLW}). So $f$ has simple zeros at $\pm p_1, \ldots, \pm p_n$ and $\omega_1/2$.

The remaining equations $0 = g'(0) = g'''(0) = g^{(5)}(0) = \cdots$ leads to the polynomial system for $\wp(p_i)$'s:

\begin{theorem}[Type I integrability, $\rho = 4\pi(2n + 1)$] \cite[Theorem 0.4]{CLW} {\ }
\begin{itemize}
\item [(1)] For $\rho = 4\pi \ell$, $\ell = 2n + 1$. All solutions are of type I and even. 
$f$ has simple zeros at $\omega_1/2$ and $\pm p_i$ for $i = 1,
\ldots, n$, and poles $q_i = p_i + \omega_2$. 

\item [(2)] For $x_i := \wp(p_i)$, $\tilde x_i := \wp(q_i)$, and $m = 1, \ldots, n$,
$$
\sum\nolimits_{i = 1}^n x_i^m - \sum\nolimits_{i = 1}^n \tilde
x_i^m = c_m, \quad (x_m - e_2)(\tilde x_m - e_2) = \mu,
$$
for some constants $c_m$ and $\mu = (e_2 - e_1)(e_2 - e_3)$. 

\item [(3)] (Brioschi--Halphen--Crawford) There is a polynomial $p_n(B; g_2, g_3)$ of degree $n + 1$ such that $p_n(B) = 0$ if and only of the corresponding Lam\'e equation $L_{\eta = n + {1}/{2}, B}\, y = 0$ provides a type I solution. This is also equivalent to that it has finite monodromy group $M$ (in fact $PM \cong K_4$). 
\end{itemize}
\end{theorem}


\subsubsection{Type II (Scaling Family) Solutions $\Longleftrightarrow$ $\eta = n$ ($\ell = 2n$)} \label{ss:type-II} Namely,
\begin{equation} \label{e:II}
f(z + \omega_1) = e^{2i\theta_1}f(z), \qquad f(z + \omega_2) =
e^{2i\theta_2}f(z)
\end{equation}
for two real constants $\theta_1$, $\theta_2$. If $f$ satisfies \eqref{e:II} then $e^\lambda f$ also satisfies \eqref{e:II}
for any $\lambda\in \Bbb R$. Thus
$$
u_\lambda(z) = \log \frac{8 e^{2\lambda}|f'(z)|^2}{(1 +
e^{2\lambda}|f(z)|^2)^2}
$$
is a scaling family of solutions with developing maps $\{e^\lambda
f\}$, and $u_\lambda$ is a \emph{blow-up sequence}. The blow-up points for $\lambda \to \infty$
(resp.~$-\infty$) are precisely zeros (resp.~poles) of $f(z)$.

Now the logarithmic derivative $g = (\log f)'$ is elliptic on $E = \Bbb C/\Lambda$, with highest order zero at $z= 0$. Namely $z = 0$ is the only zero of $g$ and 
$$
{\rm ord}_{z = 0}\, g(z) = \ell = 2n.
$$

Again, a symmetric function argument shows that the constraint on odd derivatives $0 = g'(0) = g'''(0) = \cdots = g^{(2n -1)}(0)$ implies that $g$ is even \cite[Theorem 5.2]{CLW}. Thus $g(z)$ has zeros $\pm a_1, \cdots, \pm a_n$ and we may write
\begin{equation*} 
g(z) = \frac{\wp'(a_1)}{\wp(z) - \wp(a_1)} + \cdots +
\frac{\wp'(a_n)}{\wp(z) - \wp(a_n)}
\end{equation*}
which is further constrained by $0 = g''(0) = \cdots = g^{(2n - 2)}(0)$ ($g(0) = 0$ is automatic). This gives rise to the first $n - 1$ (algebraic) equations on $a_1, \ldots, a_n$ (see \S \ref{s:HE} for the explicit equations). And then
$$
f(z) = f(0) \exp \int_0^z g(\xi)\,d\xi
$$
should satisfy the $n$-th ``equation on monodromy''
$$
\int_{L_i} g \in \sqrt{-1}\Bbb R, \qquad i = 1, 2,
$$
where $L_1$, $L_2$ are the fundamental 1-cycles.

To study these \emph{periods integrals}, for $a \not\in \frac{1}{2}\Lambda$ we define
\begin{equation} \label{e:pI}
\begin{split}
F_i(a) &:= \int_{L_i} \frac{\wp'(a)}{\wp(\xi) - \wp(a)}\, d\xi \\
&= \int_{L_i} \Big(2\zeta(a) - \zeta(a + \xi) - \zeta(a - \xi) \Big)\, d\xi.
\end{split}
\end{equation}

\begin{lemma}[Periods integrals and critical points] \cite[Proposition 2.3]{LW2} {\ }

Let $a = r\omega_1 + s\omega_2$, then modulo $4\pi i \Bbb N$ we have
\begin{equation*}
\begin{split}
F_1(a) &= 2(\omega_1 \zeta(a) - \eta_1 a) \\
&= 2(\zeta(a) - r\eta_1 - s\eta_2)\omega_1 - 4\pi i s,\\
F_2(a) &= 2(\omega_2 \zeta(a) - \eta_2 a) \\
&= 2(\zeta(a) - r\eta_1 - s\eta_2)\omega_2 + 4\pi i r.
\end{split}
\end{equation*}
In particular, the ``''monodromy equation'' is equivalent to a non-holomorphic equation in gradients of the Green function:
$$
\int_{L_i} g\,d\xi = \sum_{j = 1}^n F_i(a_j) \in \sqrt{-1} \Bbb R, \quad i = 1, 2 \quad \Longleftrightarrow \quad \sum_{j = 1}^n \nabla G(a_j) = 0.
$$
\end{lemma}

In the simplest case $\rho = 8\pi$ ($n = 1, \ell = 2$), $a_1 = p$, $a_2 = -p$,
$$
f(z) = f(0) \exp \int_0^z \frac{\wp'(p)}{\wp(\xi) - \wp(p)}\,d\xi
$$
leads to a solution if and only if  $F_i(p) \in \sqrt{-1}\,\mathbb{R}$ for $i = 1, 2$ which is then equivalent to the critical point equation $\nabla G(p) = 0$.

\begin{theorem} \cite[Theorem 1.2]{LW} \label{t:unique}
For $\rho = 8\pi$, the mean field equation \eqref{e:mfe1} on a flat torus has at most one solution up
to scaling.
\end{theorem}

\begin{corollary} \cite{LW} \label{c:3-5}
The Green function has either 3 or 5 critical points.
\end{corollary}

The first known proof of Corollary \ref{c:3-5} is by way of the uniqueness theorem in Theorem \ref{t:unique} which uses techniques in non-linear analysis including the Moser symmetrization and sharp isoperimetric inequalities. Recently, two elementary proofs were found together with more precise understanding on the distribution of critical points when the moduli point $\tau$ varies. 

Our approach uses \emph{pre-modular forms} \cite{CKLW}, and another approach uses \emph{anti-holomorphic dynamics} \cite{BE}. The former will be presented in the next section.  

\subsection{Geometry of critical points over $\mathcal{M}_{1, 1}$}

\begin{theorem}[Moduli dependence] \cite[Theorem 1.2]{CKLW} \label{t:moduli}{\ }
\begin{itemize}
\item[(1)] Let $\Omega_3 \subset \mathcal{M}_{1, 1}\cup \{\infty\}\cong
S^2$ (resp.\ $\Omega_5$) be the set of tori with 3 (resp.\ 5)
critical points, then $\Omega_3 \cup\{\infty\}$ is closed
containing $\sqrt{-1} \mathbb{R}$, $\Omega_5$ is open containing the
vertical line $[e^{\pi i/3}, \sqrt{-1} \infty)$. \smallskip

\item[(2)] Both $\Omega_3$ and $\Omega_5$ are simply connected
with $C := \p\Omega_3 = \p\Omega_5$ being homeomorphic to $S^1$
containing $\infty$. \smallskip

\item[(3)] Moreover, the extra critical points are split out from
some degenerate half period point when the tori move from $\Omega_3$ to $\Omega_5$ across $C$. \smallskip

\item[(4)] (Strong uniqueness) The map $\Omega_5 \to [0, 1]^2$ by $\tau \mapsto (r, s)$ for $p(\tau) = r\omega_1 + s\omega_2$ is a bijection onto 
$$
\triangle = [(\tfrac{1}{3}, \tfrac{1}{3}), (\tfrac{1}{2}, \tfrac{1}{2}), (0, \tfrac{1}{2})].
$$
\end{itemize}
\end{theorem}

\begin{figure}
\renewcommand{\figurename}{Figure}
\begin{center}
\includegraphics[width=0.8\textwidth]{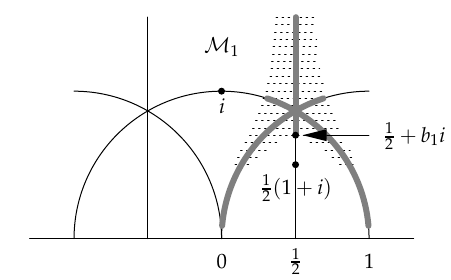}
\end{center}
\caption{$\Omega_5$ contains a neighborhood of $e^{\pi i/3}$.}
\end{figure}

On the line ${\rm Re}\, \tau = 1/2$ which are equivalent to the rhombuses tori, the proof was first achieved in \cite[Theorem 1.6]{LW} by \emph{functional equations} of $\T_1$. The general case uses the theory of \emph{pre-modular forms} which is based on Hecke's \emph{modular forms of weight one} \cite{Hecke}. I will start with a sketch of the idea. 

\begin{figure} 
\renewcommand{\figurename}{Figure}
\begin{center}
\includegraphics[width=0.7\textwidth]{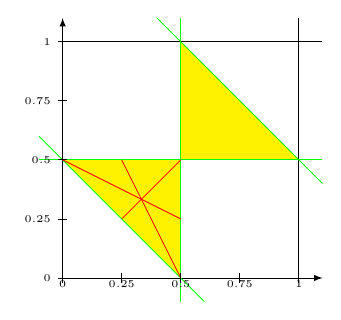}
\end{center}
\caption{The triangle $\triangle$ which is in bijection with $\Omega_5$.}
\end{figure}

Consider Hecke's weight one modular function for $\Gamma(N)$:
\begin{equation*}
\begin{split}
Z_{N, k_1, k_2}(\tau) &:= \zeta\Big( \frac{k_1 \omega_1 + k_2
\omega_2}{N}; \tau\Big) -  \frac{k_1 \eta_1 + k_2 \eta_2}{N}\\
&= -Z_{N, N - k_1, N - k_2}(\tau).
\end{split}
\end{equation*}
Define the analogue of the Euler $\phi$ function by
$$
\Psi(N) := \#\{\,(k_1, k_2)\mid (N, k_1, k_2) = 1, 0 \le k_i \le N
- 1\,\},
$$
and then we have the weight $\Psi(N)$ version for the full modular group:
$$
Z_N(\tau) := \prod_{(N, k_1, k_2) = 1} Z_{N, k_1,
k_2}(\tau) \in M_{\Psi(N)}({\rm SL(2, \Bbb Z)}).
$$

For each $\tau \in \Bbb H$ with $Z_N(\tau) = 0$, it is a double zero (at least generically). Also for odd $N \ge 5$, it is clear that $\nu_i(Z_N) = \nu_\rho(Z_N) = 0$. At $\tau = \infty$, Hecke calculated the asymptotic expansion of $Z_N$ and get 
$$
\nu_\infty(Z_N) = \phi(N/2) = 0.
$$
Then the degree formula for modular forms says that
$$
(Z_N)_{\rm red} = \frac{1}{2} \deg Z_N = \frac{1}{2} \sum_p \nu_p(Z_N) =
\frac{\Psi(N)}{24}.
$$
Let $N$ be a large prime number, then $\Psi(N)/24 = (N^2 - 1)/24$ is close to the area of the triangle (see Figure 2) 
$$
\triangle := [(\tfrac{1}{3}, \tfrac{1}{3}),
(\tfrac{1}{2}, \tfrac{1}{2}), (0, \tfrac{1}{2})].
$$ 
This suggests strongly an 1-1 correspondence between $\Omega_5$ and $\triangle$ under the map $\tau \mapsto (r, s)$ where $p(\tau) = r\omega_1 + s\omega_2$.

The actual proof uses deformations in $r, s \not \in \tfrac{1}{2} \Bbb Z$ to make the above idea rigorous. Let $F \subset \Bbb H$ be the fundamental domain for $\Gamma_0(2)$ defined by 
$$
F := \{\tau \in \Bbb H \mid 0 \le {\rm Re}\, \tau \le 1, \, |\tau - \tfrac{1}{2}| \ge \tfrac{1}{2} \}.
$$

\begin{figure}
\renewcommand{\figurename}{Figure}
\begin{center}
\includegraphics[width=\textwidth]{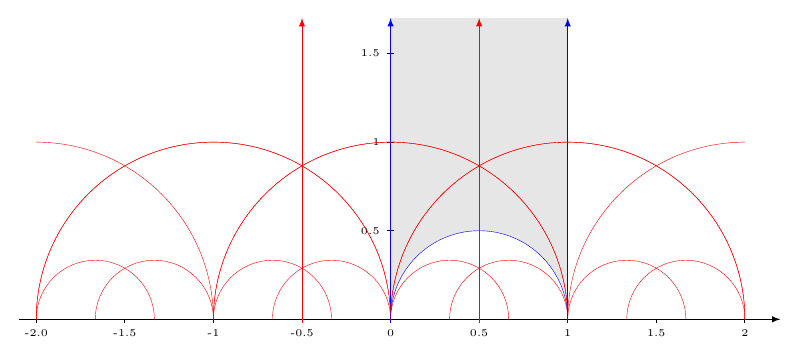}
\end{center}
\caption{The fundamental domain for $\Gamma_0(2)$.}
\end{figure}
We analyze solutions $\tau \in F$ for $Z_{r, s}(\tau) = 0$ by varying $(r, s)$.

For $\tau \in \p F$, $E$ is a rectangular torus and the only critical points of $G$ are half periods (cf.~Example \ref{e:rect}). So $Z_{r, s}(\tau) \ne 0$ for $\tau \in \p F$. 

Based on this, we use of the argument principle along the curve $\p F$ to analyze the number of zeros of $Z_{r, s}$ in $F$. Using Jacobi's triple product formula \cite[\S 21.41]{Whittaker} we may deduce that
\begin{equation*} \label{q-Z}
\begin{split}
Z_{r, s}(\tau) &= 2\pi i (s - \tfrac{1}{2}) - \pi i \frac{2 e^{2\pi i z}}{1 - e^{2\pi i z}} \\
&\qquad - 2\pi i \sum_{n = 1}^\infty \left( \frac{e^{2\pi i z} q^n}{1 - e^{2\pi i z} q^n} - \frac{e^{-2\pi i z} q^n}{1 - e^{-2\pi i z} q^n} \right),
\end{split}
\end{equation*}
where $z = r + s \tau$. From here, it follows easily that

\begin{lemma} [Asymptotic behavior of $Z_{r, s}$ on cusps] \cite[(5.6), (5.8)]{CKLW}\label{l:cusp} {\ }

We have $Z_{t, s}(-1/\tau) = \tau Z_{-s, t}(\tau)$, and for $t \in (0, 1)$,
\begin{equation*} \label{t=1/2}
Z_{r, s}(\tau) = \frac{-1}{\tau} Z_{-s, r}(-1/\tau) = \frac{2\pi i}{\tau} \big( \tfrac{1}{2} - r + o(1) \big)
\end{equation*}
as $\tau \to 0$. Similarly, $Z_{r, s}(\tau + 1) = Z_{r + s, r}(\tau)$, and for $r + s \in (0, 1)$,
\begin{equation*} 
Z_{r, s}(\tau) = Z_{r + s, s}(\tau - 1) = \frac{2\pi i}{\tau - 1} \big( \tfrac{1}{2} - (r + s) + o(1) \big).
\end{equation*}
\end{lemma}

\begin{lemma} [Non-Vanishing] \cite[Lemma 5.2]{CKLW} \label{no-sol} {\ }

For any $\tau \in \Bbb H$,
\begin{itemize}
\item[(i)] $\zeta(\tfrac{3}{4}\omega_1 + \tfrac{1}{4} \omega_2)) \ne \tfrac{3}{4} \eta_1 + \tfrac{1}{4} \eta_2$.

\item[(ii)] $\zeta(\tfrac{1}{6}\omega_1 + \tfrac{1}{6} \omega_2)) \ne \tfrac{1}{6} \eta_1 + \tfrac{1}{6} \eta_2$.

\end{itemize}
\end{lemma}

This follows from the addition law. For example, for (ii) we choose $z = \tfrac{1}{6}(\omega_1 + \omega_2) = \tfrac{1}{6} \omega_3$ and $u = \tfrac{1}{3} \omega_3$. Then
\begin{equation*}
\begin{split}
0 &\ne \frac{\wp'(z)}{\wp(z) - \wp(u)} = \zeta(\tfrac{1}{2} \omega_3) + \zeta(-\tfrac{1}{6}\omega_3) - 2\zeta(\tfrac{1}{6} \omega_3) \\
&= \tfrac{1}{2} \eta_1 + \tfrac{1}{2} \eta_2 - 3 \zeta(\tfrac{1}{6} \omega_3) = -3 (\zeta(\tfrac{1}{6} \omega_1 + \tfrac{1}{6}\omega_2) - \tfrac{1}{6}\eta_1 - \tfrac{1}{6} \eta_2).
\end{split}
\end{equation*}

It remains to show that if
$$
(r, s) \in [0, 1] \times [0, \tfrac{1}{2}]\backslash \{(0, 0), (\tfrac{1}{2}, 0), (0, \tfrac{1}{2}), (\tfrac{1}{2}, \tfrac{1}{2})\}
$$
then $Z_{r, s}(\tau) = 0$ has a solution $\tau \in \Bbb H$ if and only if that 
$$
(r, s) \in \triangle := \{(r, s) \mid 0 < r, s < \tfrac{1}{2},\, r + s > \tfrac{1}{2} \}.
$$ 
Moreover, the solution $\tau \in F$ is unique for any $(r, s) \in \triangle$.

For the proof, notice that the cases $(t, s) \not\in \triangle$ are excluded by Lemma \ref{l:cusp} and Lemma \ref{no-sol}. From
$$
\nu_{\infty}(Z_{3}) + \frac{1}{2} \nu_i(Z_{3}) + \frac{1}{3} \nu_\rho(Z_{3}) + \sum_{p \ne \infty, i, \rho} \nu_p(Z_{3}) = \frac{8}{12},
$$  
we see that $Z_{\frac{1}{3}, \frac{1}{3}}(\rho) = Z_{\frac{2}{3}, \frac{2}{3}}(\rho) = 0$ implies $\nu_\rho(Z_{(3)}) = 2$ and all other terms $= 0$. Thus $\tau = \rho$ is a simple root to $Z_{\frac{1}{3}, \frac{1}{3}}(\tau) = 0$ and such a property holds for all $(r, s) \in \triangle$ by Lemma \ref{l:cusp}. 

This completes the sketch of proof of Theorem \ref{t:moduli}.

\section{Correspondence $\Bbb P^1 \leftarrow \overline X_n \to E$ and pre-modular forms $Z_n$} \label{s:HE}

We continue the discussion on type II solutions for general $n \in \Bbb N$:

\begin{theorem} [Periods integrals and type II solutions] \cite[Theorem 0.6]{CLW} {\ }

Consider the mean field equation $\triangle u + e^u = \rho\, \delta_0$ on $E = \Bbb C/\Lambda$.
\begin{itemize}
\item[(i)] If solutions exist for $\rho = 8n \pi$, then there is a unique
even solution within each type II scaling family. ($\ell = 2n$, $a_{n + i} = -a_i$ for $i = 1, \ldots, n$.) The solution $u$ is determined by the zeros $a_1,
\ldots, a_n$ of $f$ through the recipe
\begin{equation*}
\begin{split}
g(z) &= \sum_{i = 1}^n \frac{\wp'(a_i)}{\wp(z) - \wp(a_i)}, \\ 
f(z) &= f(0) \exp \int^z g(\xi)\, d\xi.
\end{split}
\end{equation*}

\item[(ii)] The concentration of the zero of $g$ at $z \in \Lambda$: 
$$
{\rm ord}_{z = 0}\, g(z) = 2n
$$ 
leads to $n - 1$ algebraic equations for $a = \{a_1, \ldots, a_n\}$.

\item[(iii)] The $n$-th equation is given by $\int_{L_i} g \in \sqrt{-1}\Bbb R$, which is equivalent to 
$$
\sum_{i = 1}^n \nabla G(a_i) = 0.
$$
\end{itemize}
\end{theorem}

\subsection{Liouville and Lam\'e curves} \label{ss:LL}

We start by describing the $n - 1$ algebraic equations which defines an algebraic curve $X_n \subset {\rm Sym}^n E$, called the Liouville curve in \cite{CLW}. It turns out to be the unramified loci of another hyperelliptc curve $Y_n \to \Bbb C$, called the Lam\'e curve, which parametrizes log-free solutions of Lame equations.

Under the notations $(w, x_j, y_j) = (\wp(z), \wp(a_j), \wp'(a_j))$, we compute
\begin{equation*}
\begin{split}
g(z) &= \sum_{j = 1}^n \frac{1}{w} \frac{y_j}{1 - x_j/w} \\
&= \sum_{j = 1}^n \frac{y_j}{w} + \sum_{j = 1}^n \frac{y_j
x_j}{w^2} + \cdots + \sum_{j = 1}^n \frac{y_j x_j^r}{w^{r + 1}} +
\cdots.
\end{split}
\end{equation*}
Since $g(z)$ has a zero at $z = 0$ of order $2n$ and $1/w$ has a
zero at $z = 0$ of order two, we get $n - 1$ vanishing conditions:
$$
\sum_{j = 1}^n y_j x_j^r = \sum_{j = 1}^n \wp'(a_j) \wp(a_j)^r = 0, \quad 0 \le r \le n - 2.
$$


\begin{theorem} [Hyperelliptic geometry \cite{Dahmen, CLW}]
For $x_i := \wp(a_i)$, $y_i := \wp'(a_i)$, the $n - 1$ algebraic equations 
$$
\sum y_i x_i^r = 0, \quad r = 0, \ldots, n - 2,
$$ 
defines a 2 to 1 map $a \mapsto \sum \wp(a_i)$:
$$
X_n \subset {\rm Sym}^n E \longrightarrow \Bbb P^1, \qquad \{(x_i, y_i)\} \mapsto \sum x_i.
$$
\end{theorem}

The proof relies on its relation to Lam\'e equations:
\begin{equation*}
\begin{split}
f &= \exp \int g\,dz = \exp \int \sum_{i = 1}^n (2 \zeta(a_i) -
\zeta(a_i - z) -
\zeta(a_i + z))\,dz\\
&= e^{2\sum_{i = 1}^n \zeta(a_i)z} \prod_{i = 1}^n \frac{\sigma(z
- a_i)}{\sigma(z + a_i)} = (-1)^n \frac{w_a}{w_{-a}},
\end{split}
\end{equation*}
where $\displaystyle w_a(z) := e^{z\sum \zeta(a_i)}
\prod_{i = 1}^n \frac{\sigma(z - a_i)}{\sigma(z) \sigma(a_i)}$ is the Hermite--Halphen ansatz.

\begin{lemma} \cite[Theorem 6.5, Corollary 6.7]{CLW}
We have $a \in X_n$ if and only if $w_a$ and $w_{-a}$ are two independent solutions of the Lam\'e equation
\begin{equation*}
\frac{d^2 w}{dz^2} - \Big(n(n + 1)\wp(z) + (2n -
1)\sum\nolimits_{i = 1}^n \wp(a_i)\Big) w = 0.
\end{equation*}
Namely the map $a \mapsto B_a$ is given by $B_a = (2n - 1) \sum \wp(a_i)$ .
\end{lemma}

This is a long calculation which will not be repeated here. Instead, I will review the classical argument leading to the explicit hyperelliptic model $C^2 = \ell_n(B; g_2, g_3)$ (cf.~\cite[\S 23.7]{Whittaker}, \cite[\S 7.3.1]{CLW}). 

Let $w_a (z)$ and $w_{-a} (z)$ be the ansatz solutions to the Lam\'e equation
$$
w'' = (n(n + 1) \wp(z) + B) w.
$$
Let $X(z) = w_a(z) w_{-a}(z)$. By the addition theorem,
\begin{equation*}
X(z) = (-1)^n \prod_{i = 1}^n \frac{\sigma(z + a_i) \sigma (z - a_i)}{\sigma(z)^2 \sigma(a_i)^2} = (-1)^n \prod_{i = 1}^n (\wp(z) - \wp(a_i)).
\end{equation*}
That is, $X(x) = (-1)^n \prod_{i = 1}^n (x - x_i)$ is a polynomial in $x$. By construction, $X(z)$ satisfies the second symmetric power of the Lam\'e equation:
\begin{equation*}
\frac{d^3X}{dz^3} - 4 (n (n + 1) \wp + B) \frac{dX}{dz} -2n(n + 1) \wp' X = 0.
\end{equation*}
Hence $X(x)$ is a polynomial solution, in variable $x$, to
\begin{equation*} 
p(x) X''' + \tfrac{3}{2} p'(x) X'' - 4((n^2 + n - 3)x + B) X' - 2n (n + 1) X  = 0.
\end{equation*}
Thus $X$ is determined by $B$ and certain initial conditions.

Write $X(x) = (-1)^n (x^n - s_1 x^{n - 1} + \cdots + (-1)^n s_n)$, this translates to a linear recursive relation for $\mu = 0, \ldots, n -1$: 
\begin{equation*} 
\begin{split}
0 = 2(n - \mu)(2\mu + 1)(n + \mu + 1) &s_{n - \mu} \\
\quad - 4(\mu + 1) B &s_{n - \mu - 1} \\
\qquad + \tfrac{1}{2} g_2 (\mu + 1) (\mu + 2) (2\mu + 3) &s_{n - \mu - 2}\\
\quad \qquad - g_3 (\mu + 1) (\mu + 2) (\mu + 3) &s_{n - \mu - 3}.
\end{split}
\end{equation*}

We set $s_0 = 1$. For $\mu = n - 1$ we get $B = (2n - 1) s_1$ as expected. Thus all $s_2, \ldots, s_n$, $X(z)$, are determined by $s_1$, i.e.~by $B$, alone. A slightly more work shows that the set $a = \{a_i\}$ is also determined by $B$ up to sign. Hence $a \mapsto B_a$ is 2 to 1. A complete description of the hyperelliptic geometry is given in \cite[Theorem 0.7]{CLW}. We list only some of its properties:
\begin{itemize}
\item[(i)] 
There natural projective compactification $\overline X_n \subset {\rm Sym}^n E$ coincides with the projective model of the hyperelliptic curve $Y_n$ defined by
\begin{equation*}
\begin{split}
C^2 = \ell_n(B, g_2, g_3) = 4B s_n^2 + 4g_3 s_{n - 2} s_n - g_2 s_{n - 1} s_n - g_3 s_{n - 1}^2
\end{split}
\end{equation*}
in affine coordinates $(B, C)$, where $s_k = s_k(B, g_2, g_3) = r_k B^k + \cdots \in \Bbb Q[B, g_2, g_3]$ is an universal polynomial of homogeneous degree $k$ with $\deg g_2 = 2$, $\deg g_3 = 3$, and $B = (2n - 1) s_1$. 

\item[(ii)] In particular, $\overline X_n$ coincides with the unramified loci of $Y_n \to \Bbb C$: $(B, C) \mapsto B$, and the added infinity point $0^n \in \overline X_n$ is a smooth point.

\item[(iii)] Thus $\deg \ell_n = 2n + 1$ and $\overline X_n$ has arithmetic genus $g = n$. 

\item[(iv)] $\overline X_n$ is smooth except for a finite number of $\tau$, namely the discriminant loci of $\ell_n(B, g_2, g_3)$ so that $\ell_n(B)$ has multiple roots. In particular $\overline X_n$ is smooth for rectangular tori.
\end{itemize}

\subsection{Pre-modular forms} A closer look at the Hecke function $Z$ leads to  

\begin{definition} \cite[Definition 0.1]{LW-II}
An analytic function $h$ in $(z, \tau) \in \Bbb C \times \Bbb H$ is pre-modular of weight $k \in \Bbb N$ if it satisfies
\begin{itemize}
\item[(1)] For any fixed $\tau$, h(z) is analytic in $z$ and $\bar z$ and it depends only on $z \pmod{\Lambda_\tau} \in E_\tau$; 

\item[(2)] For any fixed torsion type $z \pmod{\Lambda_\tau} \in E_\tau[N]$, the function $h(\tau)$ is modular of weight $k$ with respect to $\Gamma(N)$.
\end{itemize}
\end{definition}

Now we are ready to study the last equation on $\overline X_n$:
\begin{equation} \label{Green-eq}
0 = -4\pi \sum\nolimits_{i = 1}^n \nabla G(a_i) = \sum\nolimits_{i = 1}^n Z(a_i).
\end{equation}

Consider the fundamental rational function on $E^n$:
$$
\z_n(a_1, \ldots, a_n) := \zeta (a_1 + \cdots + a_n) - \sum\nolimits_{i = 1}^n \zeta(a_i).
$$
Let $a_i = r_i \omega_1 + s_i \omega_2$, then 
\begin{equation*}
\begin{split}
-4\pi \sum \nabla G(a_i) &= \sum (\zeta(a_i) - r_i \eta_1 - s_i \eta_2) \\
&= \zeta(\sum a_i) - (\sum r_i) \eta_1 - (\sum s_i) \eta_2 - \z_n(a) \\
&= Z(\sum a_i) - \z_n(a).
\end{split}
\end{equation*} 
Hence (\ref{Green-eq}) is equivalent to
\begin{equation} \label{z=Z}
\z_n(a) = Z(\sum a_i).
\end{equation}
It is thus crucial to study the branched covering (addition) map
\begin{equation} \label{e:sigma}
\sigma_n: \overline X_n \to E, \qquad a \mapsto \sigma_n(a) := \sum_{i = 1}^n a_i.
\end{equation}

\begin{theorem} \cite[Theorem 0.2, Theorem 0.3]{LW-II} \label{t:Wn} {\ }
\begin{itemize}
\item [(1)] $\deg \sigma_n = \tfrac{1}{2} n (n + 1)$.

\item [(2)] There is a universal (weighted homogeneous) polynomial 
$$
W_n(\z) \in \Bbb C[g_2, g_3, \wp(\sigma), \wp'(\sigma)][\z]
$$ 
of degree $\tfrac{1}{2} n (n + 1)$ with
$$
W_n(\z_n) = 0.
$$
Moreover, $\z_n \in K(\overline X_n)$ is a primitive generator for the field extension $K(\bar X_n)$ over $K(E)$. 

\item [(3)] The function 
$$
Z_n(\sigma; \tau) := W_n(Z)
$$ 
is pre-modular of weight $\tfrac{1}{2} n (n + 1)$. That is, it is $\Gamma(N)$-modular if $\sigma \in E_\tau[N]$.
\end{itemize}
\end{theorem}

Idea of proof for (1): apply \emph{Theorem of the Cube} in the theory of abelian varieties \cite{Mumford}: for any three morphisms 
$$
f, g, h: V_n \longrightarrow E
$$ 
and $L \in {\rm Pic}\,E$, there is an isomorphism
\begin{equation*}
\begin{split}
(f + g + h)^*L &\cong (f + g)^*L \otimes (g + h)^*L \otimes (h + f)^*L \\
&\qquad \otimes f^*L^{-1} \otimes g^*L^{-1} \otimes h^*L^{-1}.
\end{split}
\end{equation*}

We apply it to the case $V_n \subset E^n$ which is the ordered $n$-tuples so that $V_n/S_n = \bar X_n$, and $\deg L = 1$. We prove inductively that the map 
$$
f_k(a) := a_1 + \cdots + a_k
$$
has degree $\tfrac{1}{2}k(k + 1) n!$.
 
This is non-trivial for $k = 1, 2$. It makes use that $\infty \in \overline X_n$ is non-singular and requires a detailed classification of Lam\'e functions in 4 species \cite{Whittaker}. 

From $k$ to $k + 1$ with $k \ge 2$: let 
$$
f(a) = f_{k - 1}(a), \qquad g(a) = a_{k}, \qquad h(a) = a_{k + 1}.
$$ 
Then $f_{k + 1}$ has degree $n!$ times
\begin{equation*}
\begin{split}
&\tfrac{1}{2}k(k + 1) + 3 + \tfrac{1}{2} k (k + 1) - \tfrac{1}{2} (k - 1)k - 1 - 1 \\
&= \tfrac{1}{2} (k + 1)(k + 2)
\end{split}
\end{equation*}
as expected.

Idea of proof of (2): the major tool is the \emph{tensor product} of two Lam\'e equations $w'' = I_1 w$ and $w' = I_2 w$, where $I = n(n + 1) \wp(z)$, 
$$
I_1 = I + B_a, \qquad I_2 = I + B_b.
$$ 

For $\bar X_n(\tau)$ smooth, and for a general point $\sigma_0 \in E$, we need to show that the $\tfrac{1}{2} n (n + 1)$ points on the fiber of $\bar X_n \to E$ above $\sigma_0$ has distinct $\z_n$ values. 

It is enough to show that for $\sigma_n(a) = \sigma_n(b) = \sigma_0$, 
$$
\sum \zeta(a_i) = \sum \zeta(b_i) \Longrightarrow B_a = B_b
$$ 
(and then $a = b$ by the hyperelliptic property of $\overline X_n$). 

If $w_1'' = I_1 w_1$ and $w_2'' = I_2 w_2$, then the product $q = w_1 w_2$ satisfies 
\begin{equation*} \label{intro-4th-ode}
q'''' -2(I_1 + I_2) q'' -6I'q' + ((B_a - B_b)^2 - 2I'')q = 0.
\end{equation*}

If $a = b$, a third ODE (the second symmetric power) is enough, as is studied in \S \ref{ss:LL} in deriving the hyperelliptic property..

If $a \ne b$, by addition law we find that 
$$
Q = w_a w_{-b} + w_{-a} w_b
$$ 
is an \emph{even elliptic function} solution. It is a \emph{polynomial} in $x = \wp(z)$. This leads to strong constraints on the corresponding 4-th order ODE in variable $x$, and eventually leads to a contradiction except for a finite choices of $\sigma_0$. The actual proof in \cite{LW-II} is lengthy and will not be repeated here. 

Proof for (3) is immediate: since $Z$ is pre-modular of weight one, it follows that $Z_n(\sigma; \tau) := W_n(Z)$ is pre-modular of weight $\frac{1}{2} n (n + 1)$.

\begin{example} [$n = 2$]
For $\z_2(a_1, a_2) = \zeta(a_1 + a_2) - \zeta(a_1) - \zeta(a_2)$, on $X_2$:
$$
\z_2^3(a) - 3\wp(a_1 + a_2) \z_2(a) - \wp'(a_1 + a_2) = 0.
$$
On $E^2$ it has one more term $-\tfrac{1}{2} (\wp'(a_1) + \wp'(a_2))$. Thus,
$$
Z_2(\sigma; \tau) = W_2(Z) = Z^3 - 3\wp(\sigma) Z - \wp'(\sigma).
$$
\end{example}

\begin{example} [$n = 3$] 
For $\z = \z_3(a) = \zeta(a_1 + a_2 + a_3) - \zeta(a_1) - \zeta(a_2) - \zeta(a_3)$, on $X_3$:
$$
\z^6 - 15\wp \z^4 - 20 \wp'\z^3 + (\tfrac{27}{4} g_2 - 45 \wp^2) \z^2 - 12\wp' \wp \z - \tfrac{5}{4} \wp'^2 = 0.
$$
Thus, $Z_3(\sigma; \tau) = W_3(Z)$.
\end{example}

Both $Z_2$ and $Z_3$ are known to \cite{Dahmen} (cf.~Appendix \ref{classical-Z}). In principle, since the addition map $\sigma_n$ can be explicitly computed, the polynomial $W_n(\bz)$ can be calculated from certain resultant via the elimination theory. Unfortunately the computation is necessarily demanding. A slight short cut was observed in \cite{LW-II} based on relations with the finite gap integration theory and some explicit formulas in \cite{Maier} on the \emph{twisted Lam\'e polynomials}. This allows to produce some new examples. A couple hours \emph{Mathematica} calculation gives:

\begin{example} [$n = 4$] \cite[Example 5.10]{LW-II} \label{ex:W4}
The degree 10 polynomial $W_4(\bz)$ is 
\begin{equation} \label{W4} 
\begin{split}
W_4(\bz) &= \bz^{10} - 45 \wp \bz^8 - 120 \wp' \bz^7 + (\tfrac{399}{4}g_2 - 630 \wp^2) \bz^6 -504 \wp \wp' \bz^5 \\
&\quad  - \tfrac{15}{4} (280 \wp^3 - 49 g_2 \wp - 115 g_3) \bz^4 + 15(11 g_2 - 24 \wp^2) \wp' \bz^3\\
&\qquad  - \tfrac{9}{4} (140 \wp^4 - 245 g_2 \wp^2 + 190 g_3 \wp + 21 g_2^2) \bz^2 \\
&\qquad \quad -(40 \wp^3 - 163 g_2 \wp + 125 g_3) \wp' \bz + \tfrac{3}{4}(25 g_2 - 3\wp^2) (\wp')^2.
\end{split} 
\end{equation}
The weight 10 pre-modular form $Z_4(\sigma; \tau)$ is then obtained.
\end{example}

For $n = 1$, $Z_1 \equiv Z = -4\pi \nabla G$ is the Hecke modular function. The critical point equation ($\Longleftrightarrow$ type II solutions of MFE) is transformed into zero of pre-modular forms. This now generalizes to all $n \ge 1$:

\begin{theorem} \cite{LW-II} \label{t:equiv} The following are all equivalent: \begin{itemize} 
\item [(i)] Solution $u$ to the mean field equation for $\rho = 8\pi n$. 

\item [(ii)] Periods integrals 
$$
\displaystyle \int_{L_j} g \in \sqrt{-1} \Bbb R
$$ 
for $j =1, 2$. (The real coefficients give the $\omega_j$ coordinates of $\sum a_i$.)

\item [(iii)] Green function equation on $X_n$:
$$
\sum_{i = 1}^n \nabla G(a_i) = 0,
$$ 
where $X_n$ is the unramified loci of the hyperelliptic curve $Y_n$. 

\item [(iv)] Coincidence equation under the addition map $\sigma_n: \overline X_n \to E$:
$$
\z_n(a) = Z(\sigma_n(a)).
$$

\item [(v)] Non-trivial zero of the pre-modular form 
$$
Z_n(\sigma; \tau) := W_n(Z).
$$
\end{itemize}
\end{theorem}

For the last one, we notice that the branch point $a \in Y_n \backslash X_n$ ($a = -a$) satisfies the Green equation trivially, which are excluded.

\subsection{\bf Chamber structure and wall crossing} 

\subsubsection{Deformations in $\rho \in \Bbb R^+$}

In \cite{LW}, the PDE technique used in $\rho = 8\pi$ is the \emph{method of continuity} to connect to the unique solution for $\rho = 4\pi$ by establishing the non-degeneracy of linearized equations over $[4\pi, 8\pi]$. 

This requires Moser's symmetrization and isoperimetric inequalities in singular metrics. For general $\rho$, such a non-degeneracy statement is wrong. Nevertheless, since solutions $u_\eta$ always exist for $\rho = 8\pi \eta$, $\eta \not\in \Bbb N$, it is natural to study 
$$
\lim_{\eta \to n} u_\eta.
$$ 

If the limit does not blow up, it converges to a solution $u$ for $\rho = 8\pi n$. 

For the blow-up case, we have the connection between the blow-up set and the hyperelliptic geometry of $\overline X_n \to \Bbb P^1$:

\begin{theorem} \cite[Theorem 0.7.5]{CLW} \label{blow-up-set}
Suppose that $S = \{a_1, \cdots, a_n\}$ is the blow-up set of a sequence of solutions $u_k$ with $\rho_k \to 8\pi n$ as $k \to \infty$, then 
$$
S\in Y_n = \overline X_n \setminus \{\infty\}.
$$ 
Moreover, 
\begin{itemize}
\item[(1)] If $\rho_k \ne 8\pi n$ then $S$ is a branch point ($a = -a$) of $Y_n \to \Bbb C$.

\item[(2)] If $\rho_k = 8\pi n$ for all $k$, then $S$ is not a branch point of $Y_n$.
\end{itemize}
\end{theorem}

\begin{figure}
\renewcommand{\figurename}{Figure}
\begin{center}
\includegraphics[width=0.95\textwidth]{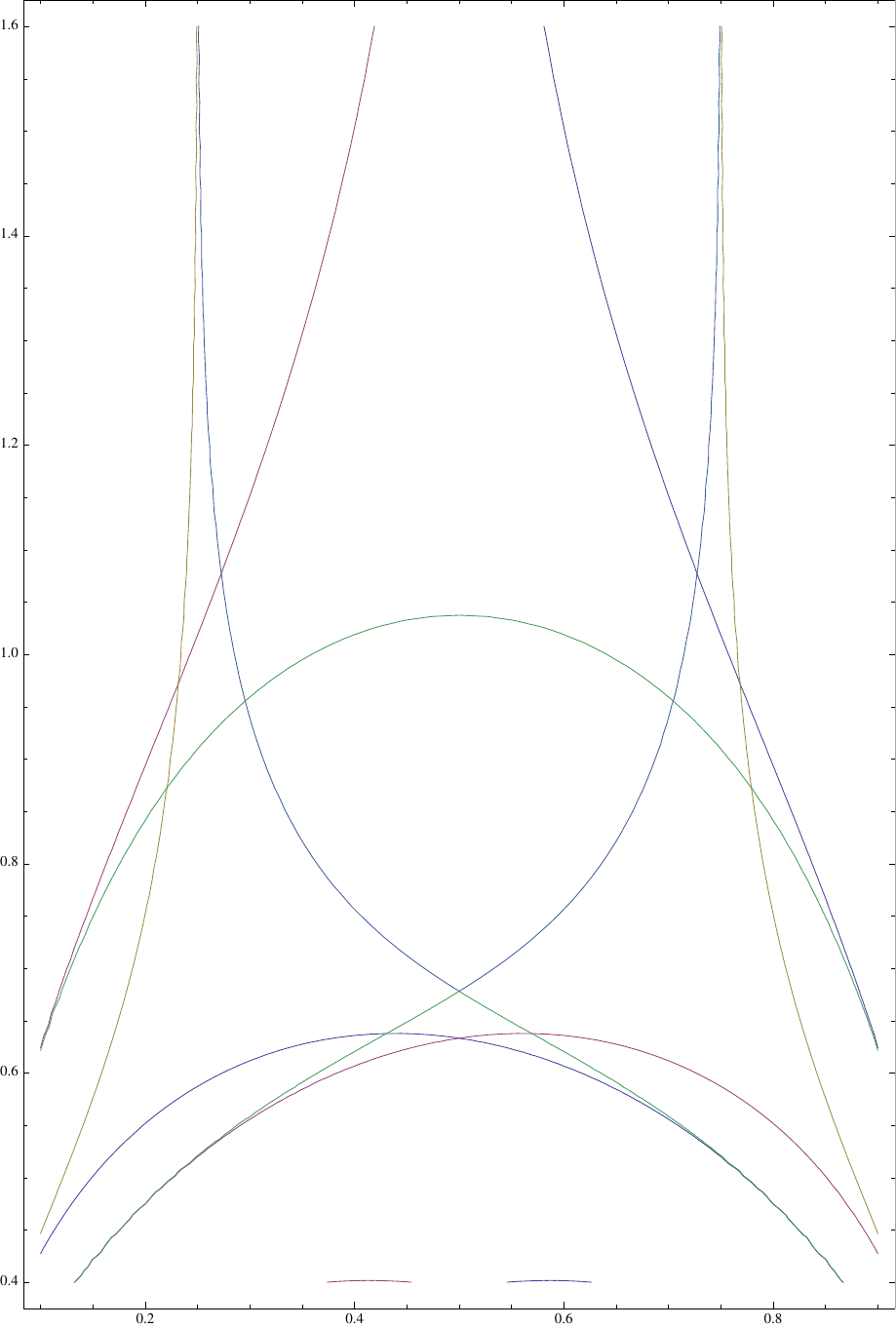}
\end{center}
\caption{$n = 2$: the degenerate curves of branch points for $\rho = 16 \pi$.}
\end{figure}

\begin{figure}
\renewcommand{\figurename}{Figure}
\begin{center}
\includegraphics[width=0.9\textwidth]{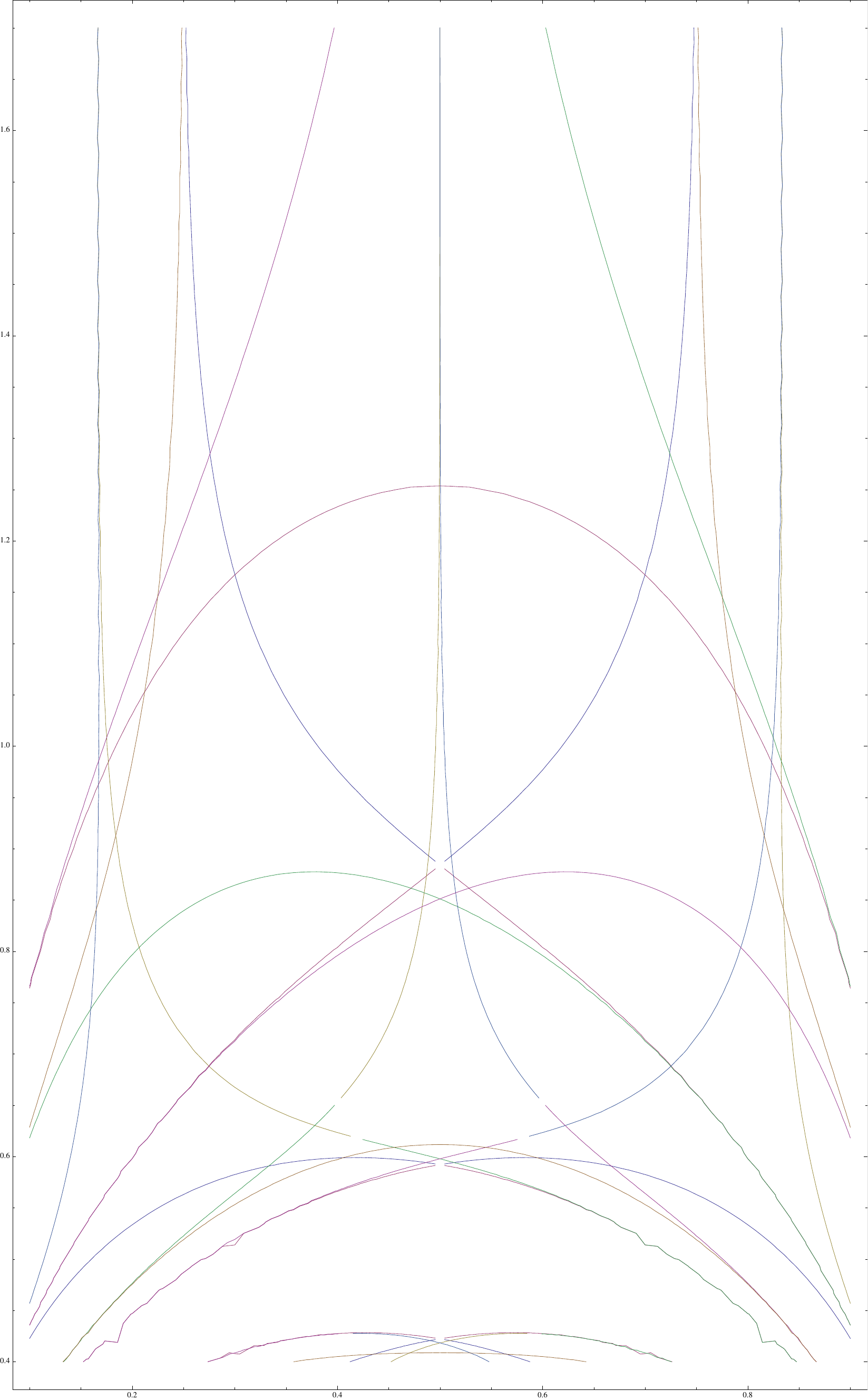}
\end{center}
\caption{$n = 3$: the degenerate curves of branch points for $\rho = 24 \pi$.}
\end{figure}

\subsubsection{Deformations in $\tau \in \mathcal{M}_{1, 1}$} For $n = 1$, the branch points are precisely the half-periods. Their degenerate loci $C$ (as degenerate critical points of the Green function) form the common boundary
$$
C := \p\Omega_3 = \p\Omega_5.
$$

For $n > 1$, similar idea applies to analyze the number of solutions of mean field equations when $\tau$ crosses the \emph{non-singular points of the degenerate curves} associated to a branch point $a \in Y_n \setminus X_n$. 

Recently it was proved in \cite{CL} that equation \eqref{e:mfe1} on \emph{rectangular tori} for $\rho = 8 n \pi$ has no solutions. Thus the chamber containing $\sqrt{-1} \Bbb R$ could be taken as the initial chamber to start the deformation argument as in \S \ref{modular}.

\section{Polynomial systems via Schwarzian derivatives} \label{s:poly-sys}
\setcounter{equation}{0}

We study equation (\ref{MFE-ln}) and its corresponding generalized Lam\'e equation (\ref{e:gLame}). The parameter $(\{A_i\}, B)$ is constrained so that the solutions to (\ref{e:gLame}) are all log-free. This is equivalent to that $f := w_1/w_2$ is log-free for one (and then any) choice of independent solutions $w_1$ and $w_2$. 

Throughout the paper we use the notations that 
$$
\wp_{ij} := \wp(p_i - p_j) \quad \mbox{and} \quad \zeta_{ij} := \zeta(p_i - p_j)
$$
whenever $i \ne j$.

\subsection{Recursive relation for logarithmic-free solutions}

The connection between equations \eqref{MFE-ln} and \eqref{e:gLame} is established via the Schwarzian derivative 
$$
S(f) := \frac{f'''}{f'} - \frac{3}{2}\Big(\frac{f''}{f'}\Big)^2 
$$ 
of the developing map $f$. Indeed, let $v = \log f'$, then
\begin{equation} \label{g-mS-der}
\begin{split}
S(f) &= v'' - \tfrac{1}{2} (v')^2 = u_{zz} - \tfrac{1}{2} u_z^2 \\
&= -2\Big(\sum_{i = 1}^{N} \eta_i (\eta_i + 1)\wp(z - p_i) + \sum_{i = 1}^{N} A_i \zeta(z - p_i) + B\Big),
\end{split}
\end{equation}
for some constants $A_i, B \in \Bbb C$ and $\eta_i = \frac{1}{2} \ell_i \in \frac{1}{2} \Bbb N$. (Since $S(f)$ is defined on $T$ and meromorphic with double poles on $S$.) The periodic constraint gives 
\begin{equation} \label{g-periodic}
F_0 := \sum_{i = 1}^{N} A_i = 0.
\end{equation}

The local expansion at $p_i$ reads as
\begin{equation} \label{g-exps}
\begin{split}
f(z) &= c_{i, 0} + c_{i, \ell_i + 1} (z - p_i)^{\ell_i + 1} + \cdots, \\
v(z) &= \log (\ell_i + 1)c_{i, \ell_i + 1} + \ell_i\log (z - p_i) + \sum_{k \ge 1} d_{i, k} (z - p_i)^k, \\
v'(z) &= \frac{\ell_i}{z - p_i} + \sum_{k \ge 0} e_{i, k} (z - p_i)^k, \quad e_{i, k} := (k + 1) d_{i, k + 1}.
\end{split}
\end{equation}

We compare the coefficients in (\ref{g-mS-der}) in this expansion. The $(z - p_i)^{-2}$ terms match automatically since $-\ell_i - \tfrac{1}{2} \ell_i^2 = -\tfrac{1}{2} \ell_i (\ell_i + 2) = -2 \eta_i (\eta_i + 1)$. 

For $(z - p_i)^{-1}$, we get 
$$
\ell_i e_{i, 0} = 2A_i \quad\Longrightarrow \quad e_{i, 0} = \frac{2A_i}{\ell_i}.  
$$

For $(z - p_i)^0$, i.e.~constant terms,
$$
e_{i, 1} - \tfrac{1}{2} 2 \ell_i e_{i, 1} - \tfrac{1}{2} e_{i, 0}^2 =  -2 \sum_{j \ne i} (\eta_j (\eta_j + 1) \wp_{ij} + \zeta_{ij} A_j) - 2B,
$$
hence
\begin{equation} \label{ei1}
(\ell_i - 1) e_{i, 1} = -2 \Big(\frac{A_i^2}{\ell_i^2} - B - \sum_{j \ne i} (\zeta_{ij} A_j + \eta_j(\eta_j + 1) \wp_{ij})\Big).
\end{equation}

If $\ell_i = 1$, this leads to a quadratic equation: 
\begin{equation} \label{g-quadratic}
A_i^2 = B + \sum_{j \ne i} \zeta_{ij} A_j + \tfrac{3}{4} \wp_{ij}.
\end{equation}
Otherwise (\ref{ei1}) detrmines $e_{i, 1}$.

Let $k \ge 1$. In general for the $(z - p_i)^k$ terms,
\begin{equation} \label{rec-k}
(k + 1 - \ell_i) e_{i, k + 1} - \tfrac{1}{2} \sum_{t = 0}^k e_{i, t} e_{i, k - t} = -2 \sum_{j} (\zeta^{(k)}_{ij} A_j + \eta_j (\eta_j + 1) \wp^{(k)}_{ij}).
\end{equation}
The right hand side is a linear polynomial $-2 L_k$ in $A_j$'s, where $h^{(k)}_{ij}$ denotes the $k$-th Taylor coefficient of $h_j(z)$ at $z = p_i$. Notice that $\wp^{(k)}_{ii} = 0$ for $k \le 1$ and $\zeta^{(k)}_{ii} = 0$ for $k \le 2$. 

\begin{lemma}
All terms in the LHS and RHS are modular functions of weight $k + 2$ if we formally assign weight one to $A_j$ and $\zeta$, and weight two to $B$ and $\wp$. 
\end{lemma}

\begin{proof}
This follows inductively from (\ref{rec-k}). 
\end{proof}

To simplify notations, we denote 
$$
\tilde e_{i, s} = \tfrac{1}{2} e_{i, s} \quad \text{and} \quad \tilde A_i = \frac{A_i}{\ell_i}.
$$ 
Then $\tilde e_{i, 0} = \tilde A_i$, and (\ref{rec-k}) becomes
\begin{equation} \label{rec-k/2}
(\ell_i - (k + 1)) \tilde e_{i, k + 1} = - \sum_{t = 0}^k \tilde e_{i, t} \tilde e_{i, k - t} + L_k.
\end{equation}
For example,  
\begin{equation*}
\begin{split}
(\ell_i - 1) \tilde e_{i, 1} &= -(\tilde A_i^2 - B - L_0),\\
(\ell_i - 2) \tilde e_{i, 2} &= \frac{2}{\ell_i - 1} \tilde A_i ( \tilde A_i^2 - B - L_0) + L_1, \\
(\ell_i - 3) \tilde e_{i, 3} &=  -\frac{2}{\ell_i - 2} \frac{2}{\ell_i - 1} \tilde A_i^2 ( \tilde A_i^2 - B - L_0) - \frac{2}{\ell_i - 2} \tilde A_i L_1\\
& \qquad - \frac{1}{(\ell_i - 1)^2}( \tilde A_i^2 - B - L_0)^2 + L_2.
\end{split}
\end{equation*}
This leads to a polynomial equation 
$$
F_i := c_i A_i^{\ell_i + 1} + (\text{lower degree})= 0.
$$

\subsection{Projective completion and analysis at infinity}

For the projective completion in $\Bbb P^{N + 1} \ni [A_0: A_1: \cdots : A_N: B]$, the solution at infinity $A_0 = 0$ again leads to $A_i = 0$ for all $i$ and with $B$ free. Thus there is only one such point $Q = [0:\cdots:0:1]$. The B\'ezout degree is
$$
b_L := \prod_{i = 1}^N (\ell_i + 1).
$$

If we assume that $\ell = \sum \ell_i$ is odd, then some $\ell_i$ must also be odd and hence $b_{\vec \ell}$ is even. We will show that $Q$ is an isolated solution of the homogenized system with multiplicity $b_L/2$.

To proceed, we notice that for the extreme case that there is only one singular source at $p = p_1$, $N = 1$ and $\ell = \ell_1 = 2n + 1$, we must have $A_1 = 0$ (the linear constraint) and hence $F = F_1$ reduces to the Brioschi--Halphen polynomial $p_n(B)$ which has degree $n + 1 = \tfrac{1}{2}(\ell_1 + 1)$.

Now comes the key point: to study the highest order terms by assigning order 1 to $A_j$'s and order 2 to $B$, we may mod out the lower order terms $L_k$'s in the recursive relation. For $F_i$, only $A_i$ and $B$ are left in the top order.

\begin{example}
For small values of $\ell_i \in \Bbb N$, the dominant term $q_{\ell_i}(\{A_i\}, B)$ with order $\ell_i + 1$ of the polynomial $F_i$ is given by
\begin{equation}
\begin{split}
q_0 &= A, \\
q_1 &= -(A^2 - B),\\
q_2 &= \frac{1}{(2!)^{2}} A(A^2 - 2^2 B),\\
q_3 &= \frac{-1}{(3!)^2} (A^2 - B) (A^2 - 3^2 B), \\
q_4 &= \frac{1}{(4!)^2} A(A^2 - 2^2 B) (A^2 - 4^2 B), \\
&\cdots
\end{split}
\end{equation}
as can be checked by brute force calculations.
\end{example}

\begin{lemma} \label{top-term-lem}
Let $\ell \in \Bbb N$, $\ell \ge 2$, $\ell \tilde e_0 = A$, $(\ell - 1)\tilde e_1 = -(A^2/\ell^2 - B)$, and 
\begin{equation} \label{red-rec}
(\ell - (k + 1)) \tilde e_{k + 1} = -\sum_{t = 0}^k \tilde e_t \tilde e_{k - t}, \quad 1 \le k \le \ell - 1. 
\end{equation}
Then the critical case $k = \ell - 1$ gives 
$$
q_\ell = -\sum_{t = 0}^{\ell - 1} \tilde e_t \tilde e_{\ell - 1 - t} = \frac{(-1)^l}{(\ell!)^2} \prod_{j = 0}^\ell (A - (\ell - 2j) B^{1/2}). 
$$
\end{lemma}

Our proof of this elementary looking lemma turns out to be technical and is deferred to the next subsection. A purely combinatorial proof is highly expected though we are unable to find one at this moment.

\begin{corollary}
If $\ell = \sum_{i = 1}^N \ell_i$ is odd, then $Q$ is an isolated point at infinity. Hence all the solutions $(\{A_i\}, B)$'s are discrete points.
\end{corollary}

\begin{proof}
Consider the chart $B \ne 0$ and let $x_i = A_i/B$ for $i = 0, \ldots, N$. The dominant polynomials at $Q$ are given by
\begin{equation}
\begin{split}
\tilde f_0(x) &= \sum_{i = 1}^N x_i, \\
\tilde f_i(x) &= \prod_{j = 0}^{\ell_i} (x_i - (\ell_i - 2j) x_0^{1/2}), \quad 1 \le i \le N.
\end{split}
\end{equation}
This forces that $x_i \sim \mu_i x_0^{1/2}$ with $\mu_i \equiv \ell_i \pmod 2$. Then
$$
\sum_{i = 1}^N x_i \sim \sum_i \mu_i x_0^{1/2} \equiv \sum \ell_i x_0^{1/2} = \ell x_0^{1/2}\not\equiv 0
$$
since $\ell$ is odd.
\end{proof}

\begin{theorem} [Algebraic degree counting formula] \label{t:degree}
For $\ell = \sum_{i = 1}^N \ell_i$ being odd, the log-free parameters $(\{A_i\}, B)$ consist of a discrete set whose cardinality, counted with multiplicity, is half of the projective B\'ezout degree. Namely,
$$
a_L = \tfrac{1}{2} b_L = \tfrac{1}{2} \prod_{i = 1}^N (\ell_i + 1).
$$ 
\end{theorem}

\begin{proof}
Since the projective B\'ezout degree is given by $\prod_{i = 1}^N \deg F_i = \prod_{i = 1}^N (\ell_i + 1)$, it remains to show that the isolated point $Q$ has multiplicity given by half of it. 

To compute the multiplicity at $Q$ in the case with $\ell = \sum \ell_i$ being odd, without loss of generality we may assume that $\ell_N = 2n_N + 1$. The equation $\tilde f_N = 0$ leads to $x_N^2 = \mu^2 x_0$ for some odd integer $1 \le \mu \le \ell_N$. There are $n_N + 1 = \tfrac{1}{2}(\ell_N + 1)$ such choices.

For each choice, we substitute $x_0 = (x_l/\mu)^2$ throughout equations $\tilde f_i(x) = 0$, which is indeed a function in $x_0$ instead of in $x_0^{1/2}$, for $i = 1, \ldots, N - 1$ to get (up to a non-zero constant multiple)
$$
Q_i(x_i, x_N) = \prod_{j_i = 0}^{\ell_i} \Big(x_i - \frac{\ell_i - 2j_i}{\mu} x_N\Big) =: \prod_{j_i = 0}^{\ell_i} L_{i, j_i}(x_i, x_N).
$$

It is clear that the intersection multiplicity at $Q$ is given by the sum of multiplicities of the intersection 
$$
\bigcap\nolimits_{i = 1}^{N - 1} (L_{i, j_i} = 0)
$$ 
among all choices of linear factors $L_{i, j_i}$ in $Q_i$ parametrized by the $N - 1$ vector $\vec j = (j_1, \ldots, j_{N - 1})$. The \emph{linear intersection} is necessarily transversal since $Q$ is isolated. Hence it contributes multiplicity 1.

Finally, the total number of choices is given by 
$$
\tfrac{1}{2} (\ell_N + 1) \times \prod_{i = 1}^{N - 1} (\ell_i + 1) = \tfrac{1}{2} \prod_{i = 1}^N (\ell_i + 1).
$$
This proves the theorem.
\end{proof}

\subsection{Proof of Lemma \ref{top-term-lem}}

There are two ideas involved in the proof. The first is to go back to the non-linear ODE (\ref{g-mS-der}) coming from the Schwarzian derivative which leads to the recursive relation (\ref{rec-k}). The key point is that now we only focus on one singular point, called it $z = 0$, and we throw away all lower order terms $L_k$'s. Thus by ``defining'' $e_i = 2\tilde e_i$ and 
$$
w(z) =  \frac{\ell}{z} + \sum_{k \ge 0} e_{k} z^k,
$$ 
the recursive relation (\ref{red-rec}) ``should'' corresponds to the ODE
$$
w' - \tfrac{1}{2} w^2 = -2\Big(\frac{\ell(\ell+ 2)}{4 z^2} + \frac{A}{z} + B\Big).
$$
The problem is that for any given $\ell \in \Bbb N$, $e_k$ is defined only up to $k \le \ell - 1$ and knowing the finite Laurent polynomial $w(z)$ does not gives $q_\ell$ in a direct manner. 

Now comes the second idea. In order to read out $q_\ell$ from $w(z)$, we replace $\ell$ by a complex variable $s \in \Bbb C$ and define $e_k(s)$ and $\hat e_k(s) = -e_i(s)/2$ by the extended \emph{universal} recursive relation (the reason to introduce the additional minus sign will be clear soon)
\begin{equation}
\hat e_{k + 1}(s) = \frac{1}{s - (k + 1)} \sum_{t = 0}^k \hat e_t(s) \hat e_{k - t}(s), \qquad k \ge 1,
\end{equation}
with initial conditions 
$$
\hat e_0(s) = -\frac{A}{s} \quad \text{and} \quad \hat e_1(s) = \frac{1}{s - 1} \Big( \frac{A^2}{s^2} - B\Big).
$$
In such a manner we see that $\hat e_l(s)$ has a simple pole at $s = \ell$ and by its very definition $-q_\ell$ is precisely the residue
$$
q_\ell = -{\rm res}_{s = \ell}\, \hat e_\ell(s).
$$

For example, we calculate
$$
\hat e_2(s) = \frac{-1}{(s - 2)(s - 1)} \frac{A^2}{s^2} \Big( \frac{A^2}{s^2} - B\Big)
$$
and
$$
\hat e_3(s) = \frac{5s - 6}{(s - 3)(s - 2)(s - 1)^2 s^4} (A^2 - s^2 B) \Big( A^2 - \frac{s - 2}{5s - 6} B \Big).
$$
However, it is unclear what the general formula for $\hat e_k(s)$ should be.

As usual, we study its corresponding generating function
\begin{equation} \label{w-hat}
\hat w(z, s) =  \frac{-s}{2} z^{-1} + \sum_{k \ge 0} \hat e_{k}(s) z^k.
\end{equation}
Then $\hat w$ satisfies the following ODE in the $z$ variable:
$$
\hat w' + \hat w^2 = \frac{s(s+ 2)}{4 z^2} + \frac{A}{z} + B.
$$
This takes the form of a standard \emph{Riccati equation} and it can be transformed into a linear equation in 
$$
g(z, s) = \exp \int^z \hat w, \quad \text{i.e.}\quad \hat w = (\log g)' = \frac{g'}{g}.
$$
Indeed, 
$$
\hat w' = \frac{g''}{g} - \frac{(g')^2}{g^2} = \frac{g''}{g} - \hat w^2,
$$
hence we arrive at the following \emph{confluent hypergeometric equation} (CHG):
\begin{equation}
g'' - Q(z) g = 0, \quad \text{where} \quad Q(z) := \frac{s(s+ 2)}{4 z^2} + \frac{A}{z} + B.
\end{equation}

The Whittaker standard form of CHG equation is (c.f.~\cite{Whittaker}, Ch.XVI)
$$
W'' - \Big( \frac{4m^2 - 1}{4 z^2} - \frac{k}{z} + \frac{1}{4}\Big) W = 0.
$$
The Kummer solutions $M_{k, \pm m}(z)$ around $z = 0$ are given by
$$
M_{k, m}(z) = e^{-\frac{1}{2} z} z^{\frac{1}{2} + m} \sum_{n = 0}^\infty \frac{(-1)^n}{n!} z^n F(\tfrac{1}{2} + m - k, -n, 2m + 1; 1),
$$
where the value of the Gauss hypergeometric function at $1$ is
$$
F(a, b, c; 1) = \frac{\Gamma(c) \Gamma(c - a - b)}{\Gamma(c - b) \Gamma(c - a)}.
$$

Now we set $s = -(2m + 1)$, $A = -k$ and $B = \tfrac{1}{4}$. Then
$$
g(z, s) = e^{-\frac{1}{2} z} z^{-\frac{1}{2}s} \sum_{n = 0}^\infty \frac{(-1)^n}{n!} z^n \frac{\Gamma(-s) \Gamma(k -\tfrac{1}{2}s + n)}{\Gamma(-s + n) \Gamma(k -\tfrac{1}{2}s)}.
$$
Denote by $g_1$ the power series part. Then
$$
\hat w = \frac{g'}{g} = -\frac{1}{2} - \frac{s}{2} z^{-1} - \frac{1}{g_1} \sum_{n = 0}^\infty \frac{(-1)^n}{n!} z^n \frac{\Gamma(-s) \Gamma(k -\tfrac{1}{2}s + n + 1)}{\Gamma(-s + n + 1) \Gamma(k -\tfrac{1}{2}s)}.
$$
To read out the residue of the $z^\ell$ coefficient $\hat e_\ell(s)$, we only need to consider those $z^j$ terms in $g_1$ and $g_1'$ with $j \le \ell$. Notice that 
$$
\frac{\Gamma(-s)}{\Gamma(-s + j + 1)} = \frac{1}{(-s + j)(-s + j - 1) \cdots (-s)}
$$
is regular at $s = \ell$ unless $j = \ell$ which gives rise to a simple pole. Hence the residue is given by
\begin{equation}
\begin{split}
&-\frac{(-1)^\ell}{\ell!} \frac{(-1)^\ell}{\ell!} \frac{\Gamma(k + \frac{1}{2}\ell + 1)}{\Gamma(k - \frac{1}{2}\ell)} \\
&\quad = \frac{-1}{(\ell!)^2} \prod_{j = 0}^\ell (k + \tfrac{1}{2}(\ell - 2j) ) = \frac{(-1)^\ell}{(\ell!)^2} \prod_{j = 0}^\ell (A - \tfrac{1}{2}(\ell - 2j) ).
\end{split}
\end{equation}
This is exactly the proposed formula for $q_\ell$ since now $B = \tfrac{1}{4}$. Also it is clear that under scaling it is enough to prove Lemma \ref{top-term-lem} for any particular non-zero value of $B$. Thus the proof is complete. \smallskip

The statement that each solution $(\{A_i\}, B)$ indeed gives a type I solution to equation \eqref{MFE-ln} will be proved in Corollary \ref{c:type-I} after we develop some basic monodromy theory of equation \eqref{e:gLame}.

\subsection{The primitive case $\ell = N$}

Now we consider the primitive case of equation \eqref{MFE-ln}, namely $\ell_i = 1$ for all $i = 1, \ldots, N$, and so $\ell = N$: 
\begin{equation} \label{MFE-l}
\triangle u + e^u = 4\pi \sum_{i = 1}^{N} \delta_{p_i} \quad \text{on $T$}.
\end{equation}

We first summarize the derivation of the precise polynomial system for any $\ell \in \Bbb N$ and add a few remarks on it. 

When $\ell = 2n + 1$, we have seen that $d_\ell = 2^{2n}$. The case $\ell = 2n$ will be analyzed in more details in later sections. 

Let $f$ be the developing map of $u$ and let $v = \log f'$. Then as an elliptic function
\begin{equation} \label{mS-der}
\begin{split}
S(f) &= v'' - \tfrac{1}{2} (v')^2 = u_{zz} - \tfrac{1}{2} u_z^2 \\
&= -2\Big(\sum_{i = 1}^{\ell} \tfrac{3}{4}\wp(z - p_i) + \sum_{i = 1}^{\ell} A_i \zeta(z - p_i) + B\Big),
\end{split}
\end{equation}
for some constants $A_i$'s and $B$. (Now $\eta_i = \tfrac{1}{2}$ and $\eta_i (\eta_i + 1) = \tfrac{3}{4}$ for each $i$.) 

Recall that $\zeta_{ij} = \zeta(p_i - p_j)$, $\wp_{ij} = \wp(p_i - p_j)$, and let $\wp_i = \sum_{j \ne i} \wp_{ij}$. 

The periodic constraint gives rise to 
\begin{equation} \label{periodic}
\sum_{i = 1}^{\ell} A_i = 0.
\end{equation}

The local expansions at $p_i$ read as (now $n_i = 0$)
\begin{equation} \label{exps}
\begin{split}
f(z) &= c_{i0} + c_{i2} (z - p_i)^2 + \cdots, \\
v(z) &= \log f'(z) = \log 2c_{i2} + \log (z - p_i) + \sum_{j \ge 1} d_{ij} (z - p_i)^j, \\
v'(z) &= \frac{1}{z - p_i} + \sum_{j \ge 0} e_{ij} (z - p_i)^j, \qquad e_{ij} := (j + 1) d_{i(j + 1)}.
\end{split}
\end{equation}

We compare the coefficients in (\ref{mS-der}) in this expansion. The $(z - p_i)^{-2}$ terms match automatically. For $(z - p_i)^{-1}$, we get 
$$
e_{i0} = 2A_i.  
$$
For $(z - p_i)^0$, i.e.~constant terms,
$$
e_{i1} - \tfrac{1}{2} 2 e_{i1} - \tfrac{1}{2} e_{i0}^2 =  -\tfrac{3}{2} \wp_i -2 \sum_{j \ne i} A_j \zeta_{ij} - 2B.
$$
The $e_{i1}$ terms cancel out and we get  $\ell$ quadratic equations
\begin{equation} \label{quadratic}
A_i^2 = \sum_{j \ne i} A_j \zeta_{ij} + B + \tfrac{3}{4} \wp_i.
\end{equation}
Together with (\ref{periodic}), there are $\ell + 1$ equations on $\ell + 1$ variables 
$$
A_1, A_2, \ldots, A_{\ell} \quad \text{and} \quad B.
$$

It is natural to ask when should this system lead to a finite number of solutions? The projective completion in $\Bbb P^{\ell + 1} \ni [A_0: \cdots: A_{\ell}: B]$ has equations 
\begin{equation*}
A_i^2 = \sum_{j \ne i} A_0 A_j \zeta_{ij} + A_0 B + \tfrac{3}{4} A_0^2 \wp_i.
\end{equation*}
The additional solutions at infinity hyperplane $A_0 = 0$ leads to $A_i^2 = 0$ for all $i = 1, \ldots, \ell$ which gives one point $Q = [0: \cdots: 0: 1]$. 

We de-homogenize the equations on the chart $B \ne 0$ with coordinates $x_i = A_i/B$. The equations become $f_i(x) = 0$, $i = 0, \ldots, \ell$, where
\begin{equation} \label{infty-system}
\begin{split}
f_0(x) &:= \sum_{i = 1}^{\ell} x_i, \\
f_i(x) &:= x_i^2 - x_0 \sum_{j \ne i} \zeta_{ij} x_j - x_0 - \tfrac{3}{4} \wp_i x_0^2 \quad \text{for $i \ne 0$}.
\end{split}
\end{equation}

\begin{question}
It seems that a complete solution to the following simplest \emph{non-linear algebra} problem is not known. 

Let $f_i(x) = Q_i(x) + L_i(x)$, $i = 1, \ldots, m$ be a system of equations in $x = (x_1, \ldots, x_m)$ with $Q_i$ being quadratic and $L_i$ being linear. Let $V(f)$ be the zero loci of $f_i(x) = 0$ for all $i$. When is $x = 0$ an isolated point of $V(f)$? If it is isolated, what is the multiplicity of $0 \in V(f)$? 

Here we are seeking for a precise answer to the very special system (\ref{infty-system}) with many symmetries.  
\end{question}

\begin{remark}
Let $f_{p}$ be the type I developing map of the unique solution $u$ to
$$
\triangle u + e^u = 4\pi \delta_p \quad \text{on $T$}.
$$ 
The function $f_p$ is essentially unique up to $f_p \mapsto 1/f_p$. Then the product ansatz 
$$
f = \prod_{i = 1}^{2n + 1} f^{\pm 1}_{p_i}
$$ 
satisfies the type I relation but fails the prescribed constraints on $p_i$'s, namely 
$$
f(z) \ne c_{i,0} + c_{i,2} (z - p_i)^2 + O(|z|^2),
$$ 
due to the effect caused by the other factors. It is still a developing map for a solution to (\ref{MFE-l}) with a different set of singular data $\tilde p_i$'s where $f'(\tilde p_i) = 0$. 

The product ansatz thus induces the map $\{p_i\} \mapsto \{\tilde p_i\}$. It seems that a further study on the associated dynamical system is an important step towards a complete understanding of equation \eqref{MFE-ln} for odd $\ell$.
\end{remark}

\section{Monodromy for special generalized Lam\'e equations} \label{mono}
\setcounter{equation}{0}

\subsection{Basic monodromy theory}

Recall the generalized Lam\'e equation \eqref{e:gLame} on a torus $E$:
\begin{equation} \label{gLame}
w'' - \Big(\sum_{i = 1}^N \eta_i(\eta_i + 1) \wp(z - p_i) + \sum_{i = 1}^N A_i \zeta(z - p_i) + B\Big) w = 0
\end{equation}
with $\sum A_i = 0$. We are interested in the case that $\eta_i \in \tfrac{1}{2}\Bbb N$. Let $\ell_i = 2 \eta_i$.

At any $p_i$, the indicial equation is given by $\lambda(\lambda - 1) - \eta_i(\eta_i + 1) = 0$ and the exponents are $\lambda = -\eta_i$ and $\lambda = \eta_i + 1$. By our assumption, the exponent difference $2\eta_1 + 1 \in \Bbb N$, hence in principal the solutions might have logarithmic terms. If the solutions are all free from logarithm, then the two fundamental solutions are of the form
$$
h_1(z) = (z - p_i)^{-\eta_i} g_1(z), \qquad h_2(z) = (z - p_i)^{\eta_i + 1} g_2(z),
$$ 
where $g_1(z)$, $g_2(z)$ are holomorphic near $z = p_i$ and $g_1(p_i) = g_2(p_i) = 1$. Thus the local momodromy matrix, in this basis, is given by 
\begin{equation} \label{local-M}
\sigma_{p_i} = (-1)^{\ell_i} I_2.
\end{equation}
Clearly \eqref{local-M} then also holds in any other basis. 

Fix a point $z_0 \in E \backslash \{p_i\}$ and consider the monodromy representation
$$
\rho = \rho_{(\{p_i\}, \{A_i\}, B)}: \pi_1(E \backslash \{p_1, \ldots, p_{N}\}, z_0) \to {\rm GL}(2, \Bbb C)
$$ 
of equation (\ref{gLame}). Let $S_i = \rho(\gamma_i)$ for the standard homology basis $\gamma_1$ and $\gamma_2$. It follows from the homotopy relation that
$$
S_2^{-1} S_1^{-1} S_2 S_1 = \prod_{i = 1}^N \sigma_{p_i} = (-1)^\ell I_2
$$
where $\ell := \sum_{i = 1}^N \ell_i = \sum_{i = 1}^l 2 \eta_i$.
For $\ell$ even this reduces to $S_1 S_2 = S_2 S_1$:
$$
\rho: \pi_1(E) \cong \Bbb Z \gamma_1 \oplus \Bbb Z \gamma_2 \to {\rm GL}(2, \Bbb C).
$$
For $\ell$ odd, this reduces to the almost abelian constraint $S_1 S_2 = -S_2 S_1$:
$$
\rho: \pi_1(E) \to {\rm GL}(2, \Bbb C)/\{\pm 1\}.
$$

A basic question in representation theory is to see if $M$, the image of $\rho$, lies in the unitary subgroup. For mean field equations, we ask the similar question for the projective representation in ${\rm PGL}(2, \Bbb C)$. Recall that the Klein four-group is the non-cyclic group of order four: $K_4 = \Bbb Z/2 \times \Bbb Z/2$.  

\begin{proposition} \label{PMK4}
If $\ell$ is odd then $PM \cong K_4$. 
\end{proposition}

\begin{proof}
Let $v $ be an eigenvector of $S_1$ with $S_1 v = \lambda v$ and $\lambda \ne 0$. Then
$$
S_1 (S_2 v) = -S_2 S_1 v = -\lambda (S_2 v).
$$
That is, $-\lambda \ne \lambda$ is also an eigenvalue of $S_1$ with eigenvector $S_2 v$. In particular, with respect to the basis $v_1 = v$, $v_2 = S_2 v_1$, we have
$$
S_1 = \begin{pmatrix} \lambda & 0 \\ 0 & -\lambda \end{pmatrix}, \quad S_2 = \begin{pmatrix} 0 & b \\ 1 & 0 \end{pmatrix}.
$$
The latter one follows from the explicit matrix computation on $S_1 S_2 = -S_2 S_1$. Equivalently we may use $S_1 (S_2 v_2) = S_1 S_2^2 v = S_2^2 S_1 v = \lambda S_2^2 v = \lambda (S_2 v_2)$ to conclude that $S_2 v_2 = b v_1$ for some $b \ne 0$.

Since $S_1^2 = \lambda^2 I_2$ and $S_2^2 = b I_2$. It is clear that in ${\rm PGL}(2, \Bbb C)$ they generate the Klein four-group $K_4$.
\end{proof}

\begin{corollary} \label{c:type-I}
For $\ell$ being odd, each log-free parameter $(\{A_i\}, B)$ gives rise to a type I solution to the mean field equation (\ref{MFE-ln}).\end{corollary}

\begin{proof}
By Proposition \ref{PMK4}, each log-free parameter $(\{A_i\}, B)$ leads to a generalized Lam\'e equation with $PM \cong K_4$, hence it leads to an unique type I solution to the mean field equation. 
\end{proof}

We know little about the full monodromy group $M$. In case of one singularity it was shown in \cite{CLW} that $M$ is a finite group with $|M| = 8$. In the next section, we will prove similar result on $M$ for multiple singularities when there are symmetries on $p_i$'s. The following lemma is well-known:

\begin{lemma}
For an algebraic second order ODE $w'' = I w$, the monodromy group $M$ is finite if and only if its projective monodromy group $PM$ is finite.
\end{lemma}

\begin{proof}
Let $w_1, w_2$ be a basis of independent solutions and $f = w_1/w_2$. Then
$$
f' = \frac{w_1' w_2 - w_1 w_2'}{w_2^2} = \frac{c}{w_2^2}
$$
where $c$ is a constant since $(w_1'w_2 - w_1 w_2')' = 0$. If $PM$ is finite, i.e.~$f$ is algebraic, then $w_2$ is also algebraic, and then $w_1 = f w_2$ is also algebraic. Hence $M$ is finite. The converse is trivial.
\end{proof}

\begin{remark}
For $\ell$ being even, the relation $S_1 S_2 = S_2 S_1$ implies that they are simultaneously diagonalizable if one of them is diagonalizable. However, it can happen that none of them is diagonalizable. For example, they can be upper triangular matrices so that
$$
S_1 S_2 = \begin{pmatrix} 1 & a \\ 0 & 1\end{pmatrix} \begin{pmatrix} 1 & b \\ 0 & 1\end{pmatrix} = \begin{pmatrix} 1 & a + b\\ 0 & 1\end{pmatrix} = S_2 S_1.
$$ 
This indicates that the existence of type II solutions are more subtle.
\end{remark}

\begin{proposition} \label{no-I}
If $\ell$ is even, there will be no corresponding type I solution to the mean field equation, though it is still possible that $PM \cong K_4$. 
\end{proposition}

\begin{proof}
The matrices $S_1$ and $S_2$ of the developing map $f$ are necessarily diagonalizable. If $S_1 S_2 = S_2 S_1$ then they are simultaneously diagonalizable, hence cannot generate a type I relation.
\end{proof}

\subsection{The full monodromy in the primitive symmetric case}

In this subsection we assume that $\ell = N = 2n + 1$ and consider only the primitive case \eqref{MFE-l}. For any log-free parameter $(\{A_i\}, B)$, we associate to it the \emph{primitive generalized Lam\'e equation}
\begin{equation} \label{GLame}
w'' - \Big(\tfrac{3}{4} \sum_{i = 1}^\ell \wp(z - p_i) + \sum_{i = 1}^\ell A_i \zeta(z - p_i) + B\Big) w = 0.
\end{equation}
The condition that all solutions are free from logarithm (at all points $p_i$'s) is equivalent to that one nontrivial quotient $f = w_1/w_2$ has this property. And this is equivalent to the polynomial system (\ref{periodic}) and (\ref{quadratic}) for $A_i$'s and $B$. 

We have seen that $PM \cong K_4$ in Proposition \ref{PMK4} and we may choose $f$ to construct a type I solution to (\ref{MFE-l}). It remains to determine $M$.

We will prove a monodromy theorem for equation \eqref{GLame}, which is free from logarithmic solutions, using the \emph{method of logarithmic-free deformations} to deform the equation to a symmetric one. Namely 
$$
\frac{d^2 w}{dz^2} - H(z) w = 0 \quad\text{with}\quad H(-z) = H(z),
$$ 
which is still free from logarithmic solutions. 

\begin{lemma} \label{n-sym}
Let $\ell = 2n + 1$ with $p_i$'s all distinct, $p_{n + i} = -p_i$ for $i = 1, \ldots, n$, and $p_{2n + 1} = 0$. Then there are $2^n$ logarithmic free parameters $(\{A_i\}, B)$'s with $A_{n + i} = -A_i$ for $i = 1, \ldots, n$ and $A_{2n + 1} = 0$. That is, $H(-z) = H(z)$. 
\end{lemma}

\begin{proof}
Under the assumption $A_{n + i} =  -A_i$, (\ref{periodic}) becomes $A_{2n + 1} = 0$ and the system (\ref{quadratic}) reduces to a system of $n$ quadratic equations in $n + 1$ variables $A_1, \ldots, A_n, B$, together with a linear equation:
$$
-2 \sum_{i = 1}^n A_i \zeta(p_i) + B + \tfrac{3}{2} \sum_{i = 1}^n \wp(p_i) = A_{2n + 1}^2 = 0.
$$  
The reduced system has no infinity solutions. Hence the number of solutions counted with multiplicities coincide with its B\'ezout degree $2^n$. 
\end{proof}

\begin{remark}
For symmetric singular source divisor $L$, there are log-free parameters $(\{A_i\}, B)$'s which do not satisfy $A_{n + i} = -A_i$ in any reordering.
\end{remark}

\begin{theorem} \label{t:finite}
Let $\ell = 2n + 1$. There are $2^n$ out of the $2^{2n}$ log-free values of $(\{A_i\}, B)$, constructed from the system (\ref{periodic}) and (\ref{quadratic}), such that the $2^{n}$ type I solutions to the mean field equation (\ref{MFE-l}) are deformed from the symmetric (even) type I solutions with respect to symmetric singular source $L = \sum p_i$. 

For these $2^n$ cases, the full monodromy group $M$ is finite with $|M| = 2^{n + 3}$. 
\end{theorem}

\begin{proof}
Consider the projective monodromy representations 
$$
\rho = \rho(\{p_i\}, \{A_i\}, B): \pi_1(E \backslash \{p_1, \ldots, p_{2n + 1}\}) \to {\rm PGL}(2, \Bbb C)
$$ 
of equation (\ref{GLame}). The first simple observation is that each local projective monodromy around $p_i$ is trivial. Indeed, by \eqref{local-M}, this holds at any double pole $p$ with coefficient $\eta(\eta + 1)$, $\eta \in \tfrac{1}{2} \Bbb Z$ since we assume logarithmic freeness. Thus the representation descents to
$$
\rho = \rho(\{p_i\}, \{A_i\}, B): \pi_1(T) \to {\rm PGL}(2, \Bbb C).
$$
By the continuous dependence of ODE in its parameters, we see that $\rho$ is continuous in $p_i$'s as long as $p_i \ne p_j$ for all $i \ne j$, and with the logarithmic-freeness being preserved.

\begin{lemma} \label{K4-rigid}
Denote $\pi_1(T) = \langle a, b\rangle \cong \Bbb Z^{\oplus 2}$. If there is a $(\{p_i\}, \{A_i\}, B)$ so that ${\rm Im}\,\rho \cong K_4$, namely $\rho(a)^2 = \rho(b)^2 \equiv I_2$ with both $\rho(a)$ and $\rho(b)$ non-trivial, then this holds true for any of its logarithmic free deformations as described above. Namely, the Klein four-group representation in ${\rm PGL}(2, \Bbb C)$ is rigid.
\end{lemma}

\begin{proof}
This follows from a straightforward matrix calculation. 

Let $\mathbf A = \rho(a)$ and $\mathbf B = \rho(b)$ in ${\rm PSL}(2, \Bbb C)$. We may assume that
$$
\mathbf A = \begin{pmatrix} a & b \\ 0 & a^{-1} \end{pmatrix}, \quad
\mathbf B = \begin{pmatrix} p & q \\ r & s \end{pmatrix}.
$$
Then $\mathbf A \mathbf B = \mu \mathbf B \mathbf A$, $\mu \in \Bbb C^\times$ implies that
\begin{equation*}
\begin{split}
ap + br &= \mu ap \\
aq + bs &= \mu bp + \mu a^{-1} q \\
a^{-1} r &= \mu ar \\
a^{-1} s &= \mu br + \mu a^{-1} s.
\end{split}
\end{equation*}

Case (i) If $r = 0$ then $\mu = 1$, $s = p^{-1}$, and $b(p - p^{-1}) = q(a - a^{-1})$. 

(i)-1: If $a \ne a^{-1}$, i.e.~$a \ne \pm 1$, then $A$ can actually be diagonalized and $b = 0$, hence $q = 0$. By symmetry the case with $p \ne p^{-1}$ is handled similarly and we are left with the case: 

(i)-2: $a = \pm 1$ and $p = \pm 1$, $b$ and $q$ are arbitrary. \smallskip

Case (ii) If $r \ne 0$ then $\mu = a^{-2}$. 

(ii)-1: Again if $a \ne a^{-1}$ then we may assume that $b = 0$. Then $ap = \mu ap$ implies that $p = 0$. Hence $a^4 = 1$ and then $a = \pm i$. The last equation forces $s = 0$ and so $r = -q^{-1}$. If we rescale the bases $e_1, e_2$ to $qe_1, e_2$, then 
$$
\mathbf A = \pm i \begin{pmatrix} 1 & 0 \\ 0 & -1 \end{pmatrix}, \quad
\mathbf B = \begin{pmatrix} 0 & 1 \\ 1 & 0 \end{pmatrix}.
$$
This gives $K_4 \subset {\rm PGL}(2, \Bbb C)$.

(ii)-2: If on the other hand $a = \pm 1$ and so $\mu = 1$, we may assume $a = 1$ and then $b = 0$. Thus $\mathbf A = I_2$ and $\mathbf B$ is arbitrary.\smallskip

Notice that under such classifications, the cases (i)-1, (i)-2 and (ii)-2 are all contained in the upper triangular cone and the case (ii)-1 is the $K_4$ representation of $\Bbb Z^2$ which appears to be an isolated point and can not be deformed to the other cases. The proof is complete.
\end{proof}

By lemma \ref{n-sym}, as far as the monodromy group is concerned, we may assume that we are in the symmetric situation centtered at the origin: $p_{n + i} = -p_i$ for $1 \le i \le n$ and $p_\ell = p_{2n + 1} = 0$. We will investigate the monodromy groups in more details by descending the equation to $\Bbb P^1 \cong S^2$ under the double covering $\wp: T \to \Bbb P^1$. The argument proceeds as in \cite{BW} on Lam\'e equation with only one singular point $z = 0$.

Let $E$ be given by the Weierstrass equation under $(x, y) = (\wp(z), \wp'(z))$:
$$
y^2 = p(x) = 4x^3 - g_2 x - g_3 = 4\prod_{i = 1}^3 (x - e_i).
$$ 
We write $H(z) = h(x)$ with $h$ being meromorphic (rational) on $\Bbb P^1$. Then 
\begin{equation*}
\begin{split}
\frac{dw}{dz} &= \frac{dw}{dx} \frac{dx}{dz} = \frac{dw}{dx} \wp'(z), \\
\frac{d^2 w}{dz^2} &= \frac{d^2 w}{dx^2} \wp'(z)^2 + \frac{dw}{dx} \wp''(z) = p(x) D^2 w + \tfrac{1}{2} p'(x) D w,
\end{split}
\end{equation*}
where $D = d/d x$. Thus the equation descends to
\begin{equation}
p(x) D^2 w + \tfrac{1}{2} p'(x) Dw - h(x) w = 0 \quad \text{on $\Bbb P^1$},
\end{equation}
or equivalently,
$$
D^2 w + \frac{1}{2} \sum_{i = 1}^3 \frac{1}{x - e_i} D w - \frac{1}{4}\frac{h(x)}{\prod_{i = 1}^3 (x - e_i)} w = 0.
$$
  
By assumption and by (\ref{local-M}), all poles of $h(x)$ contribute only local monodromy $- I_2$ and in particular trivial projective monodromy (unless the pole is $e_i$ for some $i$ which is excluded in our case). Hence all projective monodromy are generated by the local momodromy matrices at $x = e_1, e_2, e_3$ and $x = \infty$. Denote them by $\sigma_1$, $\sigma_2$, $\sigma_3$, and $\sigma_\infty$ respectively. We have $\sigma_1 \sigma_2 \sigma_3 \sigma_\infty \equiv I_2 \in {\rm PGL}(2, \Bbb C)$. 

Since $h(x)$ has no poles at $e_i$'s, the indicial equation at $x = e_i$ is then
$$
\lambda(\lambda - 1) + \tfrac{1}{2} \lambda = \lambda(\lambda - \tfrac{1}{2}).
$$
This shows that $\sigma_i$ has eigenvalues $1, -1$. It is a reflection and $\sigma_i^2 = I_2$. 

\begin{remark}
The above geometric construction (double cover descent) corresponds to the algebraic decomposition $\rho(a) = \sigma_1\sigma_\infty^{-1}$ and $\rho(b) = \sigma_2\sigma_\infty^{-1}$. The abelian relation $ab = ba$ leads to $\sigma_1 \sigma_2 \sigma_\infty = \sigma_\infty \sigma_2 \sigma_1$.
\end{remark}

At $x = \infty = \wp(p_\ell)$, using $t = 1/x$ we get the indicial equation
$$
\lambda^2 - \tfrac{1}{2} \lambda - \tfrac{1}{4} \mu(\mu + 1) = (\lambda + \tfrac{1}{2} \eta)(\lambda - \tfrac{1}{2}(\eta + 1)),
$$
where $H(z) = \eta(\eta + 1)/z^2 + \cdots$. For $\eta = n_\ell + \tfrac{1}{2}$, the eigenvalues are both $-i$ for $n_\ell$ even (e.g.~$n_\ell = 0$ in the current case), and they are both $i$ for $n_\ell$ odd. In any case, $\sigma_\infty$ is a scalar multiplication of order 4 and $\sigma_\infty \equiv I_2$ in ${\rm PGL}(2, \Bbb C)$.

Thus $\sigma_1 \sigma_2 \equiv \sigma_2 \sigma_1$. Alternatively, $\sigma_1 \sigma_2 \sigma_3 \equiv I_2 \Longrightarrow \sigma_1 \sigma_2 \equiv \sigma_3$. And then $\sigma_3 = \sigma_3^{-1} = \sigma_2^{-1} \sigma_1^{-1} = \sigma_2 \sigma_1$. In particular, $\sigma_1 \sigma_2 \equiv \sigma_2 \sigma_1$. Hence $PM \cong \Bbb Z_2 \times \Bbb Z_2 = K_4$ is the abelian group generated by $\sigma_1$ and $\sigma_2$ as expected. 

It also follows that $|M| = 16 \times 2^n$ on $\Bbb P^1$ and $|M| = 8 \times 2^n$ on $E$.
\end{proof}

\begin{remark}
\begin{itemize}
\item [(1)]
Whenever $\eta \in \tfrac{1}{2}\Bbb N$, the local projective monodromy at $z = 0$ is trivial. By the uniqueness theorem for ODE, we see that the solution quotient $f = w_1/w_2$ is symmetric under $z \mapsto -z$. 
\item[(2)]
If $\eta = n_\ell \in \Bbb N$ instead, then the eigenvalues are $1, -1$ for $n_\ell$ even, and $-1, 1$ for $n_\ell$ odd. In any case $\sigma_\infty$ is a reflection and $\sigma_\infty^2 = I_2$.
\item[(3)]
If $e_i$ is a pole of $h(x)$ of order 2 with coefficient $\eta_i(\eta_i + 1)$, then the situation is analogous to the case $x = \infty$ since abstractly they are just the branch points of $\wp: E \to \Bbb P^1$. If $\eta_i = n_i + \tfrac{1}{2}$, we get $\sigma_i \equiv I_2$. If $\eta_i = n_i$, then as in the previous remark $\sigma_i$ is still a reflection. Thus, in order to get $PM \cong K_4$, we should not allow any $p_i$ to be deformed to $\tfrac{1}{2} \Lambda$ if we keep $p_i \ne p_j$ for all $i \ne j$.
\end{itemize}
\end{remark}

\section{Towards generalized Lam\'e curves} \label{IFT}
\setcounter{equation}{0}

The structure of the primitive mean field equation \eqref{MFE-l} depends on the parity of $\ell \in \Bbb N$ in a crucial manner. Hence we separate the discussions into two subsections according to the parity of $\ell$. 

\subsection{Algebraic degree for $\ell$ being odd}

Recall the following special case of Theorem \ref{t:degree} in the primitive case. 

\begin{theorem} \label{isolated Q-odd}
For primitive singular source $L$ with $\ell := \deg L = 2n + 1$ being odd, the point $Q$ is an isolated zero of (\ref{infty-system}) with multiplicity $2^{l - 1} = 2^{2n}$.
\end{theorem}

Below we will give a more transparent proof of it to motivate our later discussion on the case $\ell = 2n$. We first give its corollary:

\begin{corollary}
If $\ell = 2n + 1$ is odd then the number of solutions $(\{A_i\}, B)$ counted with multiplicities is finite and equals $2^{\ell - 1} = 2^{2n}$.
\end{corollary}

\begin{proof} [Proof of Corollary]
The B\'ezout number of the projective system is $2^{\ell}$. Hence the actual algebraic degree, namely the number of solutions of the original affine system counted with multiplicities, is given by
$$
2^{\ell} - 2^{\ell - 1} = 2^{\ell - 1}.
$$ 
Notice that the affine solutions must form a discrete set. Indeed, if there is a positive dimensional locus $Z$ of solutions in the affine piece $\Bbb C^{\ell + 1}$, it will intersect the infinity hyperplane nontrivially, namely $Q \in \bar Z$. But this contradicts to the fact that $Q$ is an isolated zero.
\end{proof}

The idea of the proof to the theorem is to use implicit function argument: since the point $Q$ has coordinates $x = 0$, we first eliminate $x_\ell$ from the system by $x_{\ell} = -(x_1 + \cdots + x_{\ell - 1})$, and then substitute $x_0$ by \emph{higher order terms} using $f_{\ell}(x) = 0$:
$$
x_0 = (x_1 + \cdots + x_{\ell - 1})^2 - x_0 \sum_{j \ne l} \zeta_{lj} x_j - \tfrac{3}{4} \wp_l x_0^2.  
$$ 
If the resulting equations $\hat f_i(x_1, \cdots, x_{\ell - 1}) = 0$ for $i = 1, \ldots, \ell - 1$ give independent equations, i.e.~$Q$ is isolated, the multiplicity at $Q$ can then be calculated in the complete local ring at $ x = 0$ as
$$
\dim \Bbb C[\![x_1, \ldots, x_{\ell - 1}]\!]/(x_1^2, \ldots, x_{\ell - 1}^2) = 2^{\ell - 1}
$$
since the representatives of the quotient ring are $\prod_{i = 1}^{\ell - 1} x_i^{m_i}$, $m_i = 0, 1$. 

More precisely, consider a neighborhood of $Q$ so that $|x_i| < \epsilon <\!\!< 1$ for all $i = 0, \ldots, \ell$. If $f_i(x) = 0$ for $i = 1, \ldots, \ell$ and $x_0 \ne 0$, then
\begin{equation} \label{e:branch}
x_i = x_0^{1/2} \Big(1 + \sum_{j \ne i} \zeta_{ij} x_j + \tfrac{3}{4} \wp_i x_0\Big)^{1/2} = x_0^{1/2} (1 + O(\epsilon)). 
\end{equation}
But then the odd-numbered sum $f_0(x) = \sum_{i = 1}^{2n + 1} x_i$ is close to a non-zero integer multiple of $x_0^{1/2}$ and is not zero. Thus $Q$ is isolated.

We shall however analyze the local structure at $Q$ more carefully so that it can also be applied to the case when $\ell = 2n$ is even. 

\begin{proof} [Proof of Theorem \ref{isolated Q-odd}]
We use the convention $\zeta_{ii} = 0$. Then $f_\ell(x) = 0$ gives
\begin{equation} \label{x_0}
x_0 = x_\ell^2 - x_0 \sum \zeta_{lj} x_j - \tfrac{3}{4}\wp_\ell x_0^2.
\end{equation}
After substitution by this expression into $f_i(x) = 0$ for $i = 1, \ldots, \ell - 1$,
\begin{equation} \label{lim-f_i}
\tilde f_i = x_i^2 - x_\ell^2 - x_0 \sum (\zeta_{ij} - \zeta_{lj}) x_j - \tfrac{3}{4} (\wp_i - \wp_l) x_0^2.
\end{equation}
By iteration, the terms in (\ref{lim-f_i}) with $x_0$ becomes of order \emph{three or more}. Moreover the terms stabilized in finite powers only if it contain $x_\ell^{2k}$ for some $k \ge 1$. At the beginning, $\tilde f_i$ is not a Morse function since its quadratic part contains only two directions $x_i$ and $x_\ell$. But in the above limit (or completion), we get $\tilde f_i = x_i^2 - x_\ell^2 + x_\ell^2 g_i(x)$. In particular its behavior is controlled by the degenerate quadric 
$$
Q_i := x_i^2 - x_\ell^2 = (x_i + x_\ell)(x_i - x_\ell) =: L_i^+ L_i^-, \quad i = 1, \ldots, \ell - 1,
$$
with the global constraint $L_\ell := x_1 + \cdots + x_\ell = 0$.

The intersections at $Q$ is linearized to the union of $2^{\ell - 1}$ intersections of $\ell$ hyperplanes
$$
L_1^{\pm} = 0, L_2^{\pm} = 0, \cdots, L_{\ell - 1}^{\pm} = 0, L_\ell = 0.
$$ 
For each fixed choice of the set of $\ell$ hyperplanes, it is clear that $x_i = \mp x_\ell$ and hence $x_i = 0$ for all $i$ when $\ell$ is odd. The intersection is transversal since otherwise there will be non-trivial kernels. The proof is completed by noting the additive property of intersection numbers. 
\end{proof}

\subsection{Integrability for $\ell$ being even}

If $\ell = 2n$ is even, the system admits non-trivial kernels (non-transversal) if and only if there are precisely $n - 1$ equations $L_i^-$'s which are selected (so $x_i = x_\ell$). Moreover the kernel is one-dimensional given by (after reordering) 
$$
x_1 = \cdots = x_n = -x_{n + 1} = \cdots = -x_{2n},
$$ 
i.e., the line $t(1, \cdots, 1, -1, \cdots, -1)$, $t \in \Bbb C$.

This implies that the solutions variety $V = \{(A_1, \ldots, A_{2n}, B)\}$ is \emph{at most} a union of curves, all intersecting the infinity hyperplane at $Q$, and a finite number of points. 

There are special cases (primitive symmetric cases) where it is easy to conclude that $V$ contains curve components. 

\begin{example}[Center $o$ of $p_i$'s] \label{l=2n}
Let $\ell = 2n$ be even. 

For $n = 1$, we get $A_2 = -A_1$ and the two quadratic equations for $A_1$ and $A_2$ coincide since $\wp_{12} = \wp_{21}$. It is 
$$
A_1^2 = \zeta_{12} A_1 + B + \tfrac{3}{4} \wp_{12}.
$$

More generally, for $n \in \Bbb N$, if there is a \emph{center} $o \in T$ such that $\{p_i - o\} = \{-(p_i - o)\}$, or after reordering
$$
p_{n + i} - o = -(p_i - o)\quad \text{for} \quad i = 1, \ldots, n,
$$ 
(such a center $o$ always exists for $n = 1$) then there are special symmetric solutions with $A_{n + i} = -A_i$ for $i = 1, \ldots, n$. 

Indeed, equation (\ref{periodic}) is satisfied trivially, and the system (\ref{quadratic}) reduces to a system of $n$ quadratic equations on $n + 1$ variables $A_1, \ldots, A_n$ and $B$. This defines a curve $V_n \subset \Bbb C^n$ which is compactified to $\bar V_n \subset \Bbb P^n$ by adding the point $Q$ at infinity. 
\end{example}

Nevertheless, the curve components in $V$ might exist in the general non-symmetric cases. This requires a detailed study on the system in \eqref{infty-system}.

Let $C$ be the curve in $\Bbb C^{2n + 1}$ defined by the $2n$ quadratic equations in \eqref{infty-system} and $H$ be the hyperplane defined by the linear equation $\sum_{i = 1}^{2n} x_i = 0$ which will be called the \emph{global constraint}. To see if there are curve components in $V$ we need to analyze whether the kernel lines integrate to analytic curves near $Q$. Equivalently we ask if there are irreducible components of $C$ which lie in $H$ completely. In doing so, by \eqref{e:branch}, we take a branch of $x_0^{1/2}$ as the parameter $t$, and try to solve a curve germ of $C$ near $Q$ by (after another reordering by alternating signs for convenience):
\begin{equation} \label{e:series}
x_i(t) =(-1)^i \sum_{k = 1}^\infty a_{i, k} t^k, \qquad a_{i, 1} = 1,\qquad i = 1, \ldots, 2n,
\end{equation}
subject to the global constraints in all degrees
\begin{equation} \label{e:global}
\sum_{i = 1}^{2n} (-1)^i a_{i, k} = 0 \quad \mbox{for all $k \in \Bbb N$}. 
\end{equation}

Now (\ref{infty-system}) reads as, for $i = 1, \ldots, 2n$,
\begin{equation} \label{e:recur-k}
\Big(\sum_{k = 1}^\infty a_{i, k} t^k \Big)^2 = t^2 \sum_{j = 1}^{2n} \zeta_{ij} (-1)^j \sum_{k = 1}^\infty a_{j, k} t^k + t^2 + \tfrac{3}{4} \wp_i t^4,
\end{equation}
which leads to a recursive relation that uniquely determines $a_{i, k}$ inductively in $k$ and for all $i$, hence it determines the germ $(x_i(t))$. 

The process works for any of the $2^{2n}/2 = 2^{2n - 1}$ curve germ of $C$ near $Q$. But there are only $C^{2n}_n/2$ choices of such germs $(x_i(t))$ starting at $(x_i(0)) = Q = 0^{2n} \in \Bbb C^{2n + 1}$ whose tangent lines at $Q$ lie in $H$. So the problem is on the global constraints \eqref{e:global} for $k \ge 2$.

To simplify the appearance of various signs, we define
\begin{equation} \label{e:hat}
\hat \zeta_{ij} = (-1)^j \zeta_{ij}, \qquad \hat \zeta_i = \sum_{j} \hat \zeta_{ij},
\end{equation}
where we always set the meaningless terms to be zero: $\zeta_{ii} = 0 = \wp_{ii}$. 


The $t^2$ terms correspond trivially as both are just $t^2$. The $t^3$ terms give
\begin{equation} \label{e:a2}
2a_{i, 2} = \sum_j (-1)^j \zeta_{ij} a_{j, 1} = \hat \zeta_i.
\end{equation}
Then $2 \sum (-1)^i a_{i, 2} = \sum (-1)^{i + j} \zeta_{ij} = 0$ by the skew-symmetry of $\zeta_{ij}$. 

The $t^4$ terms give $2 a_{i, 3} + a_{i, 2}^2 = \sum_j \hat\zeta_{ij} a_{j, 2} + \tfrac{3}{4} \wp_i$, i.e.
\begin{equation} \label{e:a3}
a_{i, 3} = \tfrac{1}{4}\sum_{j, k} \hat \zeta_{ij} \hat \zeta_{jk} - \tfrac{1}{8}\hat \zeta_i^2 + \tfrac{3}{8} \wp_i.
\end{equation}
Then we require that
\begin{equation} \label{sum-ai3}
\begin{split}
&\quad \sum_i (-1)^i a_{i, 3}\\ &= \tfrac{1}{4} \sum_{i, j, k} (-1)^{i} \hat\zeta_{ij} \hat\zeta_{jk} - \tfrac{1}{8} \sum_{i, j, k} (-1)^{i} \hat\zeta_{ij} \hat\zeta_{ik} + \tfrac{3}{8} \sum_i (-1)^i \wp_i \\
&= \tfrac{3}{8} \Big(\sum_i (-1)^i\wp_i - \sum_i (-1)^i \hat\zeta_i^2 \Big)= 0,
\end{split}
\end{equation}
where we have used $\sum (-1)^{i + j + k} \zeta_{ij} \zeta_{jk} = -\sum (-1)^{i + j + k} \zeta_{ij} \zeta_{ik}$.

\begin{example} \label{e:n=2}
For $\ell = 4$ ($n = 2$), equation \eqref{sum-ai3} says that 
\begin{equation}
\begin{split}
2 (\wp_{24} - \wp_{13}) &= - (\zeta_{12} - \zeta_{13} + \zeta_{14})^2 + (-\zeta_{21} - \zeta_{23} + \zeta_{24})^2 \\
&\qquad - (-\zeta_{31} + \zeta_{32} + \zeta_{34})^2 + (-\zeta_{41} + \zeta_{42} - \zeta_{43})^2.
\end{split}
\end{equation} 
The latter is factorized into 
\begin{equation*}
\begin{split}
& (-2 \zeta_{21} - \zeta_{23} + \zeta_{24} - \zeta_{13} + \zeta_{14})(-\zeta_{23} + \zeta_{24} + \zeta_{13} - \zeta_{14}) \\
&\quad + (-\zeta_{41} + \zeta_{42} - 2 \zeta_{43} - \zeta_{31} + \zeta_{32})(-\zeta_{41} + \zeta_{42} + \zeta_{31} - \zeta_{32}) \\
&= (-2\zeta_{21} + 2 \zeta_{24} - 2 \zeta_{13} + 2 \zeta_{43}) (\zeta_{13} + \zeta_{24} - \zeta_{14} - \zeta_{23}) \\
&= 2 (\zeta_{24} - \zeta_{13} + \zeta_{12} - \zeta_{34}) (\zeta_{24} + \zeta_{13} - \zeta_{14} - \zeta_{23}).
\end{split}
\end{equation*}
Hence we require that
\begin{equation} \label{e:l=4}
\wp_{24} - \wp_{13} = (\zeta_{24} - \zeta_{13} + \zeta_{12} - \zeta_{34}) (\zeta_{24} + \zeta_{13} - \zeta_{14} - \zeta_{23}).
\end{equation}

It turn out that \eqref{e:l=4} is an identity. This can be seen by viewing both sides as elliptic functions in $z := p_2$. A Laurent/Taylor expansion shows that the principal parts (including the constant terms) of the common pole at $z = a_4$ coincide. (Here we need the fact that $(\zeta(a) + \zeta(b) + \zeta(c))^2 = \wp(a) + \wp(b) + \wp(c)$ if $a + b + c = 0$.) The extra poles at $z = p_1$ and $z = p_3$ are cancelled out by the corresponding vanishing of the other factor.

The compactified curve $\bar C \subset \Bbb P^{2n + 1}$ has degree $2^4 = 16$. It has $2^4/2 = 8$ tangent lines at $Q$. Among them, there are $C^4_2/2 = 3$ lines which lie in $H$ completely. By the identity \eqref{e:l=4}, each of the corresponding curve germs has intersection multiplicity with $H$ at $Q$ at least 4. 

We conclude that some component of $C$ lies in $H$. This follows from the fact that the total intersection multiplicity of $C$ at $Q$ is at least 
$$
(8 - 3) \times 1 + 3 \times 4 = 5 + 12 = 17 > 16 = \deg C.
$$ 
\end{example}

We remark that a straightforward, though notationally more involved, extension of the argument in Example \ref{e:n=2} shows that \eqref{sum-ai3} holds: 

\begin{lemma} \label{l:mQ>4}
The following elliptic function identity holds unconditionally:
\begin{equation} \label{ev-od}
\sum_{i} (-1)^i \wp_{i} \equiv 2 \Big(\sum_{i < j, \, \text{even}} \wp_{ij} - \sum_{i < j, \, \text{odd}} \wp_{ij}\Big) = \sum_i (-1)^i \hat\zeta_i^2. 
\end{equation} 
Hence the multiplicity of $(x_i(t))$ at $Q$ along each of the $C^{2n}_n$ choices of tangent directions is at least 4.
\end{lemma}

However this is not enough to conclude that $V$ contains non-trivial curve components $V_0 \subset C$ for any $n \ge 3$. Indeed, the conclusion will follow if 
$$
(\tfrac{1}{2} 2^{2n} - \tfrac{1}{2} C^{2n}_n) + \tfrac{1}{2} C^{2n}_n \times 4 > 2^{2n}.
$$ 
That is, if $3 C^{2n}_n > 2^{2n}$. This fails for $n \ge 3$ by a direct check.

For general $t^{2 + k}$ terms with $k \ge 3$, we have the recursive formula
\begin{equation} \label{e:ak+1}
a_{i, k + 1} = \tfrac{1}{2} \sum_j \hat\zeta_{ij} a_{j, k} - \tfrac{1}{2} \sum_{p = 2}^k a_{i, p} a_{i, k + 2 - p}.
\end{equation} 
So the global constraint is
$$
0 = \sum_i (-1)^i a_{i, k + 1} = \tfrac{1}{2} \sum_{i, j} (-1)^{i} \hat\zeta_{ij} a_{j, k} - \tfrac{1}{2} \sum_{p = 2}^k \sum_i (-1)^i a_{i, p} a_{i, k + 2 - p}.
$$

For $k = 3$ both summands in the RHS are equal since by \eqref{e:a2}
$$
-\sum_i (-1)^i a_{i, 2} a_{i, 3} = -\tfrac{1}{2}\sum_{i, j} (-1)^{i} \hat\zeta_{ij} a_{i, 3} = \tfrac{1}{2}\sum_{i, j} (-1)^{j} \hat\zeta_{ji} a_{i, 3}.
$$
Hence the equation $0 = \sum (-1)^i a_{i, 4} = \sum_{i, j} (-1)^{i} \hat\zeta_{ij} a_{j, 3}$ becomes
\begin{equation*} 
0 = \tfrac{1}{4}\sum_{i, j, k, m} (-1)^i \hat\zeta_{ij} \hat \zeta_{jk} \hat\zeta_{km} - \tfrac{1}{8} \sum_{i, j} (-1)^i \hat \zeta_{ij} \hat\zeta_i^2 + \tfrac{3}{8} \sum_{i, j} (-1)^i \hat\zeta_{ij} \wp_i.
\end{equation*}
The first term equals $\frac{1}{4} \sum_{i, j, k, m} (-1)^{i + j + k + m} \zeta_{ij} \zeta_{jk} \zeta_{km}$ which vanishes by reversing the indices $(i, j, k, m) \mapsto (m, k, j, i)$. Hence the constraint becomes 
\begin{equation} \label{sum-ai4}
\sum_{i} (-1)^i \hat\zeta_i^3 = 3\sum_{i} (-1)^i \hat\zeta_{i} \wp_i.
\end{equation}

Similarly, for $k \ge 4$, the global constraint can be written as
$$
\sum_i (-1)^i a_{i, k + 1} = \sum_{i, j} (-1)^{i} \hat\zeta_{ij} a_{j, k} - \tfrac{1}{2} \sum_{p = 3}^{k - 1}\sum_i (-1)^i a_{i, p} a_{i, k + 2 - p} = 0.
$$

\begin{question}
What are the precise conditions on $p_i$'s in solving these compatibilities equations $\sum (-1)^i a_{i, k} = 0$ for all $k$? Is equation \eqref{sum-ai4} (for $t^5$) an elliptic function identity?
\end{question}

We conclude this section by stating a conjecture on the variety $V$ even for the non-primitive case:

\begin{conjecture} \label{c:conj}
For any $L = \sum_{i = 1}^{N} \ell_i p_i$ with $\ell = \deg L$ being even. The log-free variety $V$ consists of a finite number of curves and points. Moreover, there are always non-trivial curve component $V_0 \subset V$. 
\end{conjecture}
 
With the existence of the curve component $V_0$, we may then defined the \emph{generalized Lam\'e curve} to be the branched double cover $Y_L \to V_0$ which parametrizes the log-free solutions of the generalized Lam\'e equation \eqref{gLame}. The next question is then to seek for possible extensions of the theory of pre-modular forms as in the $N = 1$ case.

\appendix

\section{A remark on the classical approach to $Z_n$} \label{classical-Z}
\setcounter{equation}{0}

A ``two stpes'' approach to the determination of $Z_2$ and $Z_3$ based on the addition law \eqref{z^2} and a classical cubic identity \eqref{z^3} of elliptic functions was developed in \cite{Dahmen}. One might hope that a more sophisticated application of the Frobenius--Stickelberger type identities (e.g.~\cite[p.458]{Whittaker}) may lead to a construction of $Z_n$ for $n \ge 4$. The purpose of this appendix is to show that this classical approach fails for all $n \ge 4$ (see Proposition \ref{n>3}).

\subsection{Explicit constructions for $n = 2, 3$} \label{explicit-mod}

We describe the first two cases $n = 2, 3$ in this subsection using a ``two steps'' procedure. 

\begin{example} [$n = 2$] \label{n=2}
For $\z = \z_2(a_1, a_2)$, we have on $X_2$:
$$
W_2(\z) = \z^3 - 3\wp(a_1 + a_2) \z - \wp'(a_1 + a_2) = 0.
$$
This is essentially equivalent to the addition law:
\begin{equation} \label{z^2}
\z^2 = \wp(a_1 + a_2) + \wp(a_1) + \wp(a_2).
\end{equation}
In particular, the weight 3 pre-modular form is 
$$
Z_2(\sigma; \tau) = W_2(Z) = Z^3(\sigma) - 3\wp(\sigma) Z(\sigma) - \wp'(\sigma).
$$
\end{example}

We start by applying the \emph{symmetrized operator} $\delta := \tfrac{1}{2}(\p_{a_1} + \p_{a_2})$ to $\z := \zeta(\sigma) - \zeta(a_1) - \zeta(a_2)$ sucessively to get
\begin{equation*}
\begin{split}
\delta \z &= \tfrac{1}{2}(\wp(a_1) + \wp(a_2)) - \wp(\sigma), \\
 \delta^2 \z &= \tfrac{1}{4}(\wp'(a_1) + \wp'(a_2)) - \wp'(\sigma).
\end{split}
\end{equation*}

We rewrite (\ref{z^2}) in the following \emph{admissible} form:
\begin{equation} \label{A2}
\z^2 = A_2(\delta\z) := 3 \wp(\sigma) + 2\delta\z.
\end{equation}
Then $2\z \delta \z  = 3\wp'(\sigma) + 2 \delta^2\z = \wp'(\sigma) + \tfrac{1}{2}(\wp'(a_1) + \wp'(a_2))$. Hence
$$
\z^3 = 3\wp(\sigma) \z + 2\z \delta\z = 3\wp(\sigma) \z + \wp'(\sigma) + \tfrac{1}{2}(\wp'(a_1) + \wp'(a_2)).
$$

On $X_2$ this reduces to $W_2(\z) = 0$. We emphasize that only (\ref{A2}) is used.

\begin{example} [$n = 3$] \label{n=3}
For $\z = \z_3(a)$, we have on $X_3$:
$$
W_3(\z) = \z^6 - 15\wp \z^4 - 20 \wp' \z^3 + (\tfrac{27}{4} g_2 - 45 \wp^2) \z^2 - 12\wp' \wp \z - \tfrac{5}{4} \wp'^2 = 0.
$$
Here $\wp$ and $\wp'$ are evaluated at $\sigma = \sum_{i = 1}^3 a_i$. 

Hence, $Z_3(\sigma; \tau) = W_3(Z)$ is of weight 6.
\end{example}

The derivation of $W_3(z)$ is more involved. It consists of two steps. The first step is the following classical identity. We supply a detailed proof of it since a variant of the proof will be used for our later discussions on the general cases $n \ge 4$.

\begin{lemma} [c.f.~{\cite[p.459]{Whittaker}}] \label{poly-lemma3}
For $\z = \zeta(\sigma) - \sum_{i = 1}^3 \zeta(a_i)$ with $\sigma = \sum_{i = 1}^3 a_i$,
\begin{equation} \label{z^3}
\z^3 = 3(\wp(\sigma) + \sum \wp(a_i)) \z + (\wp'(\sigma) - \sum \wp'(a_i)).
\end{equation}
\end{lemma}

\begin{proof}
We will prove it by viewing both sides as functions of $s = a_3$ and by comparing the principal parts on both sides. In doing so we emphasize that the case $n = 2$ is used in an essential way.

For a better presentation on signs we work on $\f = -\z_3$. Let $\sigma_2 = a_1 + a_2$ and $\f_2 = -\z_2 = \zeta(a_1) + \zeta(a_2) - \zeta(\sigma_2)$. In the following, all elliptic functions are evaluated at $s = \sigma_2$ if not written explicitly. Then
\begin{equation*}
\begin{split}
\f = \zeta(a) + \zeta(b) + \zeta(s) - \zeta(\sigma_2 + s) = \frac{1}{s} + \f_2 + \wp\, s + \tfrac{1}{2} \wp' s^2 + O(s^3),
\end{split}
\end{equation*}
by noting $\zeta(s) = 1/s + O(s^3)$.

We want to compare
\begin{equation*}
\begin{split}
\f^3 = \frac{1}{s^3} + \frac{3\f_2}{s^2} + \frac{3\f_2^2 + 3\wp}{s} +\f_2^3 + 6\f_2 \wp  + \tfrac{3}{2} \wp' + O(s)
\end{split}
\end{equation*}
with (all summations are for $i = 1, 2$) 
\begin{equation*}
\begin{split}
3&\Big(\sum \wp(a_i) + \wp(s) + \wp(\sigma_2 + s)\Big) \f + \Big(\sum \wp'(a_i) + \wp'(s) - \wp'(\sigma_2 + s)\Big) \\
&= 3\Big(\frac{1}{s^2} + (\sum \wp(a_i) + \wp) + \wp' s\Big) \Big(\frac{1}{s} + \f_2 + \wp\, s + \tfrac{1}{2} \wp' s^2\Big) \\
&\qquad + \frac{-2}{s^3} + (\sum \wp'(a_i) - \wp') + O(s) \\
&= \frac{1}{s^3} + \frac{3\f_2}{s^2} + \frac{3(\sum \wp(a_i) + \wp) + 3\wp}{s} \\
&\qquad + 3(\sum \wp(a_i) + \wp) \f_2 + \tfrac{9}{2} \wp' + (\sum \wp'(a_i) - \wp') + O(s).
\end{split}
\end{equation*} 

Now the case for $n = 2$ in (\ref{z^2}) says that $\f_2^2 = \sum \wp(a_i) + \wp$. Thus the $s^{-1}$ terms  match and the equality for the constant terms is equivalent to
\begin{equation} \label{2int}
\f_2^3 = 3\wp\, \f_2 - (\wp' + \tfrac{1}{2} \sum \wp'(a_i)).
\end{equation}
By symmetric differentiation, the $n = 2$ case implies also
\begin{equation} \label{2'}
\wp' + \tfrac{1}{2} \sum \wp'(a_i) = \delta (\f_2^2) = 2\f_2\delta \f_2 = 2\f_2 (\wp - \tfrac{1}{2}\sum \wp(a_i)).
\end{equation}
Thus the right hand side of (\ref{2int}) equals
$$
\f_2 (\wp + \sum \wp(a_i)) = \f_2^3,
$$
which is precisely the left hand side. 

This identifies the principal parts at the pole $s = 0$. For the other pole $s = -(a_1 + a_2)$ the comparison follows from the case $s = 0$ under the symmetry $s \mapsto -(a_1 + a_2) - s$. This proves the lemma.
\end{proof}

The important thing we learn from the above proof is in the last step:
The formula (\ref{2'}) is a non-monic polynomial relation for $\f_2$ which has lower degree than the monic one (whose minimal degree is 2). Though (\ref{2'}) is merely a symmetric analogue of the addition law for $\zeta$: \small
$$
\z_2 = \zeta(a_1 + a_2) - \zeta(a_1) - \zeta(a_2) = \frac{1}{2} \frac{\wp'(a_1) - \wp'(a_2)}{\wp(a_1) - \wp(a_2)},
$$ \normalsize
it allows us to deduce higher monic polynomial relations like (\ref{2int}). 

In the same way, by applying $\delta := \tfrac{1}{3}\sum_{i = 1}^3\p_{a_i}$ to (\ref{z^3}), we may compute $\delta\z^3 = 3\z^2 \delta \z$ in both ways to get a non-monic degree 2 relation \small
\begin{equation*}
\begin{split}
(\wp - \tfrac{1}{3} \sum \wp_i) \z^2 = - (\wp' - \tfrac{1}{3} \sum \wp_i') \z - \tfrac{1}{3} (\wp'' - \tfrac{1}{3} \sum \wp_i'') + (\wp + \sum \wp_i)(\wp - \tfrac{1}{3} \sum \wp_i).
\end{split}
\end{equation*} \normalsize
Here $\wp_i := \wp(a_i)$ and the sum is from $1$ to $3$.

By multiplying $\z$ to (\ref{z^3}) and replace $\z^2$ terms which involve $\sum \wp_i$ by the above relation, we get a degree 4 monic relation \small
$$
\z^4 = 12 \wp \z^2 + (10\wp' - 2\sum \wp'_i) \z + 3(\wp'' - \tfrac{1}{3} \sum \wp_i'') - 9(\wp + \sum \wp_i)(\wp - \tfrac{1}{3} \sum \wp_i).
$$ \normalsize
It is unclear if we may proceed in this way to eventually get an expression involving $\sigma$ only. Our second step towards the proof of Example \ref{n=3} is a more systematic usage of the symmetrized derivatives.

\begin{proof}[Proof of Example \ref{n=3}]
Before we apply $\delta := \tfrac{1}{3}\sum_{i = 1}^3\p_{a_i}$ to (\ref{z^3}) successively, we need to first have a good sense about the expression $\delta^r \z$ for $r \in \Bbb N$. We compute
\begin{equation} \label{d^r}
\begin{split}
\delta \z &= \tfrac{1}{3} \sum \wp(a_i) - \wp(\sigma), \\
\delta^2 \z &= \tfrac{1}{9} \sum \wp'(a_i) - \wp'(\sigma), \\
\delta^3 \z &= \tfrac{2}{9} \sum \wp^2(a_i) - \tfrac{1}{18} g_2 - \wp''(\sigma), \\
\delta^4 \z &= \tfrac{4}{27} \sum \wp(a_i) \wp'(a_i) - \wp'''(\sigma).
\end{split}
\end{equation} 
The key observation is that, $\delta^2 \z$ and $\delta^4 \z$ are good terms (which depend on $\sigma$ only) allowed in our calculations, while $\delta \z$ and $\delta^3 \z$ are terms not allowed in our final formula which need to be replaced by known terms. 

Now, in terms of $\delta^r \z$ in (\ref{d^r}), we rewrite (\ref{z^3}) as
\begin{equation} \label{zdz}
\z^3 = A_3(\delta^*\z) := 12 \wp \z + 9 \z\delta\z -8 \wp' - 9 \delta^2 \z.
\end{equation}
Here, and in the following, all $\wp$ and its derivatives are evaluated at $\sigma$. This implies that $\z\delta \z$ is indeed a sum of good terms. 

By applying $\delta$ to (\ref{zdz}) we get
\begin{equation} \label{zd^3z}
\begin{split}
3\z^2 \delta \z = 12\wp' \z + 12 \wp \delta\z + 9(\delta\z)^2 + 9\z\delta^2 \z - 8\wp'' - 9 \delta^3 \z.
\end{split}
\end{equation}
By multiplying $\z^2$ to (\ref{zd^3z}), it follows that $\z^2 \delta^3 \z$ is a sum of good terms.

By applying $\delta$ to (\ref{zd^3z}) we get
\begin{equation} \label{3-final}
\begin{split}
6\z(\delta\z)^2 + 3\z^2 \delta^2 \z &= 12 \wp'' \z + 24 \wp' \delta \z + 12 \wp \delta^2\z \\
&\quad + 27 \delta \z \delta^2 \z + 9 \z \delta^3 \z -8 \wp''' - 9 \delta^4 \z.
\end{split}
\end{equation}
By multiplying $\z$ to (\ref{3-final}) we find that all the terms appeared are now good terms. Hence it give rise to the polynomial $W_3(z)$ by explicit substitution.

In fact in this last step we have $\delta^2 \z = -\wp'$ and $\delta^4 \z = -\wp'''$ on the curve $X_3$. Thus we have
\begin{equation*}
\begin{split}
9\z \delta \z &= \z^3 - 12 \wp \z - \wp', \\
9^2\z^2 \delta^3 \z &= -2 \z^6 + 24 \wp \z^4 + 28 \wp' \z^3 - 72 \wp'' \z^2 +12 \wp \wp' \z + \wp'^2.
\end{split}
\end{equation*}
Then we get $W_3(z)$ by a straightforward manipulation with (\ref{3-final}). 
\end{proof}

\subsection{A remark on $n \ge 4$} \label{non-induction}

A closer look at the proof of Example \ref{n=3} shows that the overall important equation to start with is not the classical polynomial identity (\ref{z^3}) in $\z$. Instead, the polynomial equation (\ref{zdz}) on $\z$ and $\delta \z$ with \emph{admissible coefficients} (depending only on $\sigma$) is what we really need for the proof. (For simplicity we use the notation $A(\delta^*\z)$ for it.)

To the authors' knowledge, historically a reasonably clean polynomial identity of degree $n$ with symmetric coefficients in $a_i$'s like (\ref{z^2}) and (\ref{z^3}) was unknown for $n \ge 4$. For $n = 4$, naive generalization of the proof of Lemma \ref{poly-lemma3} to get a degree 4 polynomial fails immediately. Thus we try to get an \emph{admissible equation} in $\z$ and $\delta^i \z$'s instead. (These two type of expressions are equivalent by (\ref{d^rz}) below.) As a result, we will be able to prove that the degree 4 polynomial indeed does not exist.

Notice that while $\z = \zeta(\sum a_i) - \sum \zeta(a_i)$ is symmetric in the last variable $s = a_n$ under $s \mapsto - \sigma_n = -\sigma_{n - 1} - s$, 
$$
\delta \z = \frac{1}{n} \sum\nolimits_{i = 1}^n \wp(a_i) - \wp(\sigma_n) 
$$
breaks such a symmetry. It thus distinguishes the two poles $s = 0$ and $s = -\sigma_{n - 1}$ which is a key property we shall use. 

As usual, we have ${\rm wt}(\wp^{(j)}) = j + 2$, ${\rm wt}(\z) = 1$ and ${\rm wt}(\delta\z) = 2$. Since
\begin{equation} \label{d^rz}
\delta^{r + 1} \z = \frac{1}{n^{r + 1}} \sum\nolimits_{i = 1}^n \wp^{(r)}(a_i) - \wp^{(r)}(\sigma_n),
\end{equation}
we see that $\sum \wp^{(r)}(a_i) = n^{r + 1}(\delta^{r + 1} \z + \wp^{(r)}(\sigma_n))$ is admissible of weight $r + 2$ for all $r \ge 0$. By the elementary properties of $\wp$ function, this is equivalent to that $\sum \wp^j(a_i)$ (of weight $2j$) and $\sum \wp^j(a_i) \wp'(a_i)$ (of weight $2j + 1$), $j \ge 0$, are all admissible. (We notice that admissible terms of odd weights are \emph{good terms} when we restrict to $X_n$.) As before, we denote by $\wp$ (without variable) the admissible term $\wp(\sigma_n)$, and similarly for other elliptic functions.

We have already shown in (\ref{A2}) and (\ref{zdz}) that 
\begin{equation*}
\begin{split}
\z_2^2 &= 2\delta_2 \z_2 + 3\wp(\sigma_2), \\
\z_3^3 &= 9\z_3 \delta_3 \z_3 - 9\delta_3^2 \z_3 + 12\wp(\sigma_3) \z_3 - 8 \wp'(\sigma_3).
\end{split}
\end{equation*}
However, a naive generalization of this is not true for $n \ge 4$:

\begin{proposition} \label{n>3}
For all $n \in \Bbb N_{\ge 4}$, there is no admissible polynomial equation
$$
\z^n = A_n(\delta^*\z).
$$
Here $A_n$ is homogeneous of weight $n$ with coefficients in $\Bbb Q[\wp^{(j)}]_{j = 0, \ldots, n - 2}$. 
\end{proposition}

\begin{proof}
Let $n \ge 2$, $\z = \z_{n}$, $\sigma = \sigma_n$, $\delta= \delta_n$, and $s = a_{n + 1}$. We omit the argument in a function if it is evaluated at $\sigma$. Near the pole $s = 0$,
\begin{equation*}
\begin{split}
\z_{n + 1} &= \z - \zeta(s) + \zeta(\sigma + s) - \zeta(\sigma) \\
&= -\frac{1}{s} + \z - \wp\, s - \tfrac{1}{2} \wp' s^2 - \big( \tfrac{1}{6}\wp'' - \tfrac{1}{60}g_2 \big) s^3 + \cdots.
\end{split}
\end{equation*}
Since $\delta_{n + 1} = \tfrac{n}{n + 1} \delta + \tfrac{1}{n + 1} \p_s$, we have for $k \ge 1$
\begin{equation*}
\begin{split}
\delta_{n + 1}^{k} \z_{n + 1} &= \tfrac{n^k}{(n + 1)^k} (\delta^{k}\z + \wp^{(k - 1)}) + \tfrac{1}{(n + 1)^k}\wp^{(k - 1)}(s) - \wp^{(k - 1)}(\sigma + s) \\
&= \tfrac{(-1)^{k - 1} k !}{(n + 1)^k} \frac{1}{s^{k + 1}} + \Big(\tfrac{n^k}{(n + 1)^k} \delta^{k}\z + \big(\tfrac{n^k}{(n + 1)^k} - 1\big)\wp^{(k - 1)} + \tfrac{(k - 1)!}{(n + 1)^k} c_{k - 1}\Big) \\
&\qquad - \sum\nolimits_{j \ge 1} \Big(\tfrac{1}{j!} \wp^{(k + j - 1)} - \tfrac{(k + j - 1)!}{(n + 1)^k j!} c_{k + j - 1}\Big)s^j,
\end{split}
\end{equation*}
where $\wp(z) = z^{-2} + \sum_{j \ge 2} c_j z^j$, $c_2 = \tfrac{1}{20} g_2$, $c_4 = \tfrac{1}{28} g_3$, and $c_{2i + 1} = 0$ for all $i$. Notice that the principal part has only one term.

Now let $n = 3$, and we will prove the case for $n + 1 = 4$. We have
\begin{equation*}
\begin{split}
\z_4^2 &= \frac{1}{s^2} - \frac{2\z}{s} + (\z^2 + 2\wp) + (\wp' - 2\wp \z) s \\
&\qquad + (\wp^2 - \wp' \z + \tfrac{1}{3}\wp'' - \tfrac{1}{30}g_2) s^2 + O(s^3), \\
\delta_4\z_4 &= \tfrac{1}{4} \frac{1}{s^2} + (\tfrac{3}{4} \delta\z - \tfrac{1}{4}\wp) - \wp' s - (\tfrac{1}{2} \wp'' - \tfrac{1}{80} g_2) s^2 + O(s^3), \\
\delta_4^2 \z_4 &= -\tfrac{1}{8} \frac{1}{s^3} + (\tfrac{9}{16} \delta^2\z - \tfrac{7}{16} \wp') - (\wp'' - \tfrac{1}{160} g_2) s + O(s^2) \\
\delta_4^3 \z_4 &= \tfrac{3}{32} \frac{1}{s^4} + (\tfrac{27}{64} \delta^3\z - \tfrac{37}{64} \wp'' + \tfrac{1}{640} g_2) + O(s).
\end{split}
\end{equation*}
Hence (we still denote $\delta_4$ by $\delta$ for simplicity)
\begin{equation*}
\begin{split}
\z_4^2 \delta \z_4 &= \frac{1}{4 s^4} - \frac{\z}{2s^3} + (\tfrac{3}{4}\delta \z + \tfrac{1}{4}\z^2 + \tfrac{1}{4} \wp) \frac{1}{s^2} - (\tfrac{3}{4} \wp' + \tfrac{3}{2} \z\delta\z) \frac{1}{s} \\
&\qquad + \tfrac{3}{4} \z^2 \delta\z + \tfrac{3}{2} \wp\, \delta \z - \tfrac{1}{4} \wp \,\z^2 + \tfrac{7}{4} \wp' \z - (\tfrac{5}{12} \wp'' + \tfrac{1}{4} \wp^2 - \tfrac{1}{240} g_2) + O(s), \\
\z_4\delta^2\z_4 &= \frac{1}{8s^4} - \frac{\z}{8s^3} + \frac{\wp}{8s^2} + (\tfrac{1}{2}\wp' - \tfrac{9}{16} \delta^2\z) \frac{1}{s} \\
&\qquad + \tfrac{9}{16} \z\delta^2 \z - \tfrac{7}{16} \wp' \z + (\tfrac{49}{48} \wp'' - \tfrac{1}{120} g_2) + O(s), \\
(\delta\z_4)^2 &= \frac{1}{16s^4} + \frac{3 \delta\z - \wp}{8 s^2} - \frac{\wp'}{2 s} + (\tfrac{3}{4} \delta\z - \tfrac{1}{4}\wp)^2 - (\tfrac{1}{4}\wp'' - \tfrac{1}{160} g_2) + O(s).
\end{split}
\end{equation*}
We would like to represent
\begin{equation*}
\begin{split}
\z_4^4 &= \frac{1}{s^4} - \frac{4\z}{s^3} + \frac{6\z^2 + 4\wp}{s^2} - \frac{4\z^3 + 12\wp \z - 2\wp'}{s} \\
&\qquad + \z^4 + 12\wp \z^2 + 6(\wp^2 - \wp'\z) + (\tfrac{2}{3} \wp'' - \tfrac{1}{15} g_2) + O(s).
\end{split}
\end{equation*}
by $A_4(\delta^*\z)$, whose general form is  
$$
a \z_4^2 \delta\z_4 + b \z_4^2 \delta^2 \z_4 + c \delta^3 \z_4 + d (\delta\z_4)^2 + e\wp\, \z_4^2 + f\wp \delta \z_4 + g\wp' \z_4 + (h\wp'' + i g_2).
$$

The equations for $s^{-4}$, $s^{-3}$, $s^{-2}$ in $\z_4^4 = A_4(\delta^*\z)$ read as
\begin{equation*}
\begin{split}
1 &= \tfrac{1}{4}a + \tfrac{1}{8}b + \tfrac{3}{32}c + \tfrac{1}{16}d \quad \mbox{(for $1/s^4$)},\\
4 &= \tfrac{1}{2}a + \tfrac{1}{8}b \quad \mbox{(for $\z/s^3$)},\\
6 &= \tfrac{1}{4}a \quad \mbox{(for $\z^2/s^2$)},\\
0 &= \tfrac{3}{4}a + \tfrac{3}{8}d \quad \mbox{(for $\delta\z/s^2$)}, \\
4 &= \tfrac{1}{4}a + \tfrac{1}{8}b - \tfrac{1}{8}d + e + \tfrac{1}{4} f \quad \mbox{(for $\wp/s^2$)}.
\end{split}
\end{equation*}
The middle 3 equations give $a = 24$, $b = -64$, $d = -48$. The first equation then gives $c = 64$, an the last equation reduces to
$$
e + \tfrac{1}{4} f = 0.
$$
The equation for $s^{-1}$ then reads as $-4\z^3 - 12\wp \z + 2\wp' = (-18\wp' - 36 \z\delta \z) + (-32\wp' + 36 \delta^2 \z) + 24 \wp' - 2e\wp \z - g \wp'$. By collecting terms we get 
\begin{equation*}
\begin{split}
\z^3 = 9\z \delta \z - 9\delta^2 \z + (\tfrac{1}{2}e - 3)\wp \z + (7 + \tfrac{1}{4}g)\wp'.
\end{split}
\end{equation*}
Now we plug in the result for $\z^3$. By comparing with (\ref{zdz}) we get $e = 30$, $g = -60$. Hence $f = -120$.

The final equation for the constant term is given by
\begin{equation*}
\begin{split}
&\z^4 + 12\wp \z^2 + 6(\wp^2 - \wp'\z) + (\tfrac{2}{3} \wp'' - \tfrac{1}{15} g_2) \\
&= 24\big(\tfrac{3}{4} \z^2 \delta\z + \tfrac{3}{2} \wp\, \delta \z - \tfrac{1}{4} \wp \,\z^2 + \tfrac{7}{4} \wp' \z - (\tfrac{5}{12} \wp'' + \tfrac{1}{4} \wp^2 - \tfrac{1}{240} g_2)\big) \\
&\quad -64\big(\tfrac{9}{16} \z\delta^2 \z - \tfrac{7}{16} \wp' \z + (\tfrac{49}{48} \wp'' - \tfrac{1}{120} g_2)\big) \\
&\quad+ 64(\tfrac{27}{64} \delta^3\z - \tfrac{37}{64} \wp'' + \tfrac{1}{640} g_2) \\
&\quad -48\big( (\tfrac{3}{4} \delta\z - \tfrac{1}{4}\wp)^2 - (\tfrac{1}{4}\wp'' - \tfrac{1}{160} g_2) \big) \\
&\quad + 30\wp(\z^2 + 2\wp) -120\wp(\tfrac{3}{4}\delta \z - \tfrac{1}{4}\wp) - 60 \wp'\, \z + h\wp'' + i g_2.
\end{split}
\end{equation*}
By collecting terms we get
\begin{equation} \label{const}
\begin{split}
\z^4 &= 18 \z^2 \delta \z - 36 \z \delta^2 \z + 27 \delta^3\z - 27 (\delta\z)^2 \\
&\quad + 12\wp \z^2 - 36 \wp\delta\z + 16 \wp' \z + 75\wp^2 + (h - 101)\wp'' + (i +\tfrac{1}{2}) g_2.
\end{split}
\end{equation}

From (\ref{zdz}), by symmetric differentiation we compute
\begin{equation} \label{dz^3}
\begin{split}
3\z^2 \delta\z = 9(\delta\z)^2 + 9 \z\delta^2 \z - 9\delta^3\z + 12\wp' \z + 12 \wp \delta\z - 8\wp''.
\end{split}
\end{equation}
Multiplying (\ref{zdz}) by $\z$ and substituting $\z^2\delta\z$ by (\ref{dz^3}), we get 
\begin{equation} \label{z4-1}
\begin{split}
\z^4 &= 9\z^2 \delta\z - 9\z\delta^2\z + 12\wp\z^2 - 8\wp' \z \\
&= 18 \z\delta^2\z - 27 \delta^3\z + 27(\delta\z)^2 + 12 \wp\z^2 + 36\wp\delta\z + 28 \wp'\z - 24\wp''.
\end{split}
\end{equation}
On the other hand, by substituting $\z^2\delta\z$ in (\ref{const}) by (\ref{dz^3}), we get
\begin{equation} \label{z4-2}
\begin{split}
\z^4 &= 18 \z\delta^2 \z - 27\delta^3\z + 27(\delta\z)^2 + 12 \wp\z^2 + 36\wp\delta\z + 88\wp'\z \\
&\qquad + 75\wp^2 + (h - 149)\wp'' + (i + \tfrac{1}{2}) g_2.
\end{split}
\end{equation}
Comparing (\ref{z4-1}) and (\ref{z4-2}) we must have
$$
60 \wp'(\sigma)\z = - 99\wp^2(\sigma) - (h - 149)\wp''(\sigma) - (i + \tfrac{1}{2}) g_2.
$$
The function $\z = \zeta(a_1 + a_2 + a_3) - \zeta(a_1) - \zeta(a_2) - \zeta(a_3)$, viewed as a function in $t = a_3$, has a pole at $t = 0$. But the right hand side is clearly regular at $t = 0$. Hence a contradiction. 

This shows that for $n = 4$ no monic admissible polynomial equations of degree $n$ for $\z_n$ may exist. The proof shows also that if $A_n(\delta^*\z)$ exists for some $n \in \Bbb N$, then $A_{n - 1}(\delta^*\z)$ must also exist by looking at the residue term in the Laurent expansion of $\z_n^n = A_n(\delta^*\z_n)$. By induction this implies $A_n(\delta^*\z)$ does not exists for all $n \ge 4$.
\end{proof}

Thus in order to make Theorem \ref{t:equiv} effective we have to construct the degree $\tfrac{1}{2}n(n + 1)$ polynomial $W_n(z)$ directly without the intermediate step. This issue is now resolved in \cite{LW-II} based on the method of resultant. \medskip

We conclude the appendix by 

\begin{remark} \label{galois}
The branched cover $\overline X_n \to E$ is in general not a Galois cover. Namely, $W_n(z)$ does not split into product of linear factors in $K(\bar X_n)$. 

Indeed, for $n = 2$ we may factorize $W_2(z)$ by division to get
$$
W_2(z) = (z - \z_2)(z^2 + \z_2 z + (\z_2^2 - 3\wp(\sigma))).
$$ 
Thus the other two roots of $W_2(z)$ are given by
$$
\w_{\pm} := \frac{1}{2}\Big(-\z_2 \pm \sqrt{3(4 \wp(\sigma) - \z_2^2)}\Big).
$$
We will show that $\w_{\pm}$ are not single valued on $\overline X_2$ for general tori.

The rational function $h(a) := 4\wp(\sigma(a)) - \z_2^2(a)$ has poles of total order six by \cite[Theorem 3.13]{LW-II}. By \cite[Example 3.11]{LW-II}, if $g_2 \ne 0$, they consists of 3 poles with each of order 2, and for $g_2 = 0$, $0^2 \in \overline X_2$ is of order 2 and $(q, -q) \in \overline X_2$ with $\wp(\pm q) = 0$ is of order 4. In order for $\w_{\pm}$ being single valued, the zeros of $h$ must also be of even order.

If $h(a) = 0$ but $\wp(\sigma(a)) \ne 0$ for some $a \in \overline X_2$, then it is easy to see that there must be some zeros of $h$ with odd order. Thus we only need to consider the case that all zeros of $h$ are also zeros of $\wp(\sigma(a))$. Since
$$
h(a) = 3\wp(a_1 + a_2) - (\wp(a_1) + \wp(a_2)),
$$
we have $\wp(a_1) + \wp(a_2) = 0$ too. The addition theorem then implies that $\wp'(a_1) = \wp'(a_2)$. But the equation for $X_2$ is given by $\wp'(a_1) + \wp'(a_2) = 0$, hence $\wp'(a_1) = 0 = \wp'(a_2)$. That is, $a_1 = \tfrac{1}{2} \omega_i$, $a_2 = \tfrac{1}{2} \omega_j$ and $a_1 + a_2 = \tfrac{1}{2} \omega_k$ is the third half period. Then we get the non-trivial equation $e_k = 0$. Thus for general $E = E_\tau$, $W_2(z)$ does not split into product of linear factors. 
\end{remark}

\end{document}